\documentclass[nospthms]{svjour3}

\usepackage{mathptmx}
\usepackage{helvet}
\usepackage{courier}
\usepackage{makeidx}
\usepackage{multicol}
\usepackage{footmisc}

\smartqed

\usepackage[active]{srcltx}
\usepackage{amsmath}
\usepackage{amssymb}
\usepackage{amsfonts}
\usepackage{amsthm}
\usepackage{mathdots}
\usepackage{oldgerm}

\usepackage{pifont}
\usepackage{mdwlist}
\usepackage{mathrsfs}
\usepackage{comment}
\usepackage{setspace}
\usepackage{mathdots}
\usepackage{tikz}
\usepackage{graphicx}%,pst-all,bm}
\usepackage{fixltx2e}
\usepackage{flushend}
\usepackage{subfig}
\usepackage{enumerate}
\usepackage[utf8]{inputenc}
\usepackage{caption}
\usepackage{longtable}
\usepackage{verbatim}
\usepackage{relsize}
\usepackage{exscale}
\usepackage[title]{appendix}
\usepackage{subfig}

\newcommand{\cC}{\mathcal{C}}
\newcommand{\cB}{\mathcal{B}}

\newcommand{\cU}{\mathcal{U}}
\newcommand{\cV}{\mathcal{V}}
\newcommand{\cN}{\mathcal{N}}

\newcommand{\cE}{\mathcal{E}}
\newcommand{\cK}{\mathcal{K}}
\newcommand{\cA}{\mathcal{A}}
\newcommand{\cP}{\mathcal{P}}

\newcommand{\cJ}{\mathcal{J}}

\newcommand{\cQ}{\mathcal{Q}}
\newcommand{\cX}{\mathcal{X}}
\newcommand{\cY}{\mathcal{Y}}

\makeindex             % used for the subject index
\usetikzlibrary{arrows}

\theoremstyle{theorem}
\newtheorem{Prop}{Proposition}[section]
\newtheorem{Lem}[Prop]{Lemma}
\newtheorem{Thm}[Prop]{Theorem}
\newtheorem*{Thm*}{Theorem}

\theoremstyle{definition}
\newtheorem{Def}[Prop]{Definition}
\newtheorem{Rem}[Prop]{Remark}

\newtheorem{Ex}[Prop]{Example}

\newcommand{\N}{{\mathbb{N}}}

\newcommand{\R}{{\mathbb{R}}}
\newcommand{\C}{{\mathbb{C}}}

\DeclareMathOperator{\dom}{dom}
\DeclareMathOperator{\rk}{rk}

\DeclareMathOperator{\im}{im}

\newcommand{\setdef}[2]{\left\{\ #1\ \left|\ \vphantom{#1} #2\ \right.\right\}}
\newcommand{\ddt}{\tfrac{\text{\normalfont d}}{\text{\normalfont d}t}}

\newlength{\innersep}
\newlength{\maxlength}
\newlength{\dummylength}

\newcommand{\JordanBlock}[3]{
\setlength{\arraycolsep}{0pt}
\renewcommand{\arraystretch}{0}
\settowidth{\maxlength}{$#1$}
\settoheight{\dummylength}{$#1$}
\ifdim\dummylength>\maxlength
  \setlength{\maxlength}{\dummylength}
\fi
\settowidth{\dummylength}{$#2$}
\ifdim\dummylength>\maxlength
  \setlength{\maxlength}{\dummylength}
\fi
\settoheight{\dummylength}{$#2$}
\ifdim\dummylength>\maxlength
  \setlength{\maxlength}{\dummylength}
\fi
\setlength{\innersep}{0.1\maxlength}
\addtolength{\maxlength}{\innersep}
\addtolength{\maxlength}{\innersep}
\newcommand{\invisiblebox}{\phantom{\rule{\maxlength}{\maxlength}}}
\begin{array}{ccc}
  {\tikz[remember picture] \node[outer sep=0,inner sep=\innersep] (a11) {$#1$};} & \invisiblebox & {\tikz[remember picture] \node[outer sep=0,inner sep=\innersep] (a13) {$#1$};}\\
   \invisiblebox &\phantom{\rule{#3}{#3}} & \invisiblebox \\
   {\tikz[remember picture] \node[outer sep=0,inner sep=\innersep] (a31) {$#2$};} & \invisiblebox & {\tikz[remember picture] \node[outer sep=0,inner sep=\innersep] (a33) {$#1$};}
\end{array}
\tikz[remember picture, overlay] \draw (a11) edge[very thick] (a33);
}

\newcommand{\UnityBlock}[3]{
\setlength{\arraycolsep}{0pt}
\renewcommand{\arraystretch}{0}
\settowidth{\maxlength}{$#1$}
\settoheight{\dummylength}{$#1$}
\ifdim\dummylength>\maxlength
  \setlength{\maxlength}{\dummylength}
\fi
\settowidth{\dummylength}{$#2$}
\ifdim\dummylength>\maxlength
  \setlength{\maxlength}{\dummylength}
\fi
\settoheight{\dummylength}{$#2$}
\ifdim\dummylength>\maxlength
  \setlength{\maxlength}{\dummylength}
\fi
\setlength{\innersep}{0.1\maxlength}
\addtolength{\maxlength}{\innersep}
\addtolength{\maxlength}{\innersep}
\newcommand{\invisiblebox}{\phantom{\rule{\maxlength}{\maxlength}}}
\begin{array}{ccc}
  {\tikz[remember picture] \node[outer sep=0,inner sep=\innersep] (a11) {$#1$};} & \invisiblebox & \invisiblebox\\
   \invisiblebox &\phantom{\rule{#3}{#3}} & \invisiblebox \\
   \invisiblebox & \invisiblebox & {\tikz[remember picture] \node[outer sep=0,inner sep=\innersep] (a33) {$#1$};}
\end{array}
\tikz[remember picture, overlay] \draw (a11) edge[very thick] (a33);
}

\newcommand{\ReverseUnityBlock}[3]{
\setlength{\arraycolsep}{0pt}
\renewcommand{\arraystretch}{0}
\settowidth{\maxlength}{$#1$}
\settoheight{\dummylength}{$#1$}
\ifdim\dummylength>\maxlength
  \setlength{\maxlength}{\dummylength}
\fi
\settowidth{\dummylength}{$#2$}
\ifdim\dummylength>\maxlength
  \setlength{\maxlength}{\dummylength}
\fi
\settoheight{\dummylength}{$#2$}
\ifdim\dummylength>\maxlength
  \setlength{\maxlength}{\dummylength}
\fi
\setlength{\innersep}{0.1\maxlength}
\addtolength{\maxlength}{\innersep}
\addtolength{\maxlength}{\innersep}
\newcommand{\invisiblebox}{\phantom{\rule{\maxlength}{\maxlength}}}
\begin{array}{ccc}
   \invisiblebox & \invisiblebox & {\tikz[remember picture] \node[outer sep=0,inner sep=\innersep] (a13) {$#1$};}\\
   \invisiblebox &\phantom{\rule{#3}{#3}} & \invisiblebox \\
   {\tikz[remember picture] \node[outer sep=0,inner sep=\innersep] (a31) {$#1$};} &\invisiblebox & \invisiblebox
\end{array}
\tikz[remember picture, overlay] \draw (a13) edge[very thick] (a31);
}

\newcommand{\LowerNilBlock}[3]{
\setlength{\arraycolsep}{0pt}
\renewcommand{\arraystretch}{0}
\settowidth{\maxlength}{$#1$}
\settoheight{\dummylength}{$#1$}
\ifdim\dummylength>\maxlength
  \setlength{\maxlength}{\dummylength}
\fi
\settowidth{\dummylength}{$#2$}
\ifdim\dummylength>\maxlength
  \setlength{\maxlength}{\dummylength}
\fi
\settoheight{\dummylength}{$#2$}
\ifdim\dummylength>\maxlength
  \setlength{\maxlength}{\dummylength}
\fi
\setlength{\innersep}{0.1\maxlength}
\addtolength{\maxlength}{\innersep}
\addtolength{\maxlength}{\innersep}
\newcommand{\invisiblebox}{\phantom{\rule{\maxlength}{\maxlength}}}
\begin{array}{cccc}
  \tikz[remember picture] \node[outer sep=0,inner sep=\innersep] (a11) {$#1$}; &  & \invisiblebox & \invisiblebox\\
  {\tikz[remember picture] \node[outer sep=0,inner sep=\innersep] (a21) {$#2$};}&  & \invisiblebox & \invisiblebox\\
   & \phantom{\rule{#3}{#3}} &  & \\
   \invisiblebox & &  {\tikz[remember picture] \node[outer sep=0,inner sep=\innersep] (a43) {$#2$};} &
   {\tikz[remember picture] \node[outer sep=0,inner sep=\innersep] (a44) {$#1$};}
\end{array}
\tikz[remember picture, overlay] \draw (a11) edge[very thick] (a44);
\tikz[remember picture, overlay] \draw (a21) edge[very thick] (a43);
}

\newcommand{\RectBlock}[3]{
\setlength{\arraycolsep}{0pt}
\renewcommand{\arraystretch}{0}
\settowidth{\maxlength}{$#1$}
\settoheight{\dummylength}{$#1$}
\ifdim\dummylength>\maxlength
  \setlength{\maxlength}{\dummylength}
\fi
\settowidth{\dummylength}{$#2$}
\ifdim\dummylength>\maxlength
  \setlength{\maxlength}{\dummylength}
\fi
\settoheight{\dummylength}{$#2$}
\ifdim\dummylength>\maxlength
  \setlength{\maxlength}{\dummylength}
\fi
\setlength{\innersep}{0.1\maxlength}
\addtolength{\maxlength}{\innersep}
\addtolength{\maxlength}{\innersep}
\newcommand{\invisiblebox}{\phantom{\rule{\maxlength}{\maxlength}}}
\begin{array}{ccccc}
  \tikz[remember picture] \node[outer sep=0,inner sep=\innersep] (a11) {$#1$}; &{\tikz[remember picture] \node[outer sep=0,inner sep=\innersep] (a21) {$#2$};}& & \invisiblebox & \invisiblebox\\
  &&\phantom{\rule{#3}{#3}} &&\\
   \invisiblebox &\invisiblebox& &
   {\tikz[remember picture] \node[outer sep=0,inner sep=\innersep] (a44) {$#1$};} & {\tikz[remember picture] \node[outer sep=0,inner sep=\innersep] (a43) {$#2$};}
\end{array}
\tikz[remember picture, overlay] \draw (a11) edge[very thick] (a44);
\tikz[remember picture, overlay] \draw (a21) edge[very thick] (a43);
}

\usepackage{circuitikz}

\newcommand\norm[1]{\left\lVert#1\right\rVert}

\newenvironment{smallpmatrix}%          environment name
{\left(\begin{smallmatrix}}%            begin code
{\end{smallmatrix}\right)}%             end code

\newenvironment{smallbmatrix}%          environment name
{\left[\begin{smallmatrix}}%            begin code
{\end{smallmatrix}\right]}%             end code

\makeindex

\usepackage[colorlinks=false,
linkbordercolor=red,
citebordercolor=blue]{hyperref}

\begin{document}

\title{Observers for differential-algebraic systems with Lipschitz or monotone nonlinearities\thanks{This work was supported by the German Research Foundation (Deutsche Forschungsgemeinschaft) via the grant BE 6263/1-1.}}

\author{Thomas Berger \and Lukas Lanza}

\institute{Corresponding author: Thomas Berger  \at Tel.: +49 5251 60-3779
            \and
            Thomas Berger, Lukas Lanza \at
              Institut f\"ur Mathematik, Universit\"at Paderborn, Warburger Str.~100, 33098~Paderborn, Germany \\
              \email{\{thomas.berger, lukas.lanza\}@math.upb.de}
}

\maketitle

\begin{abstract}
We study state estimation for nonlinear differential-algebraic systems, where the nonlinearity satisfies a Lipschitz condition or a generalized monotonicity condition or a combination of these. The presented observer   design unifies earlier approaches and extends the standard Luenberger type observer design. The design parameters of the observer can be obtained from the solution of a linear matrix inequality restricted to a subspace determined by the Wong sequences. Some illustrative examples and a comparative discussion are given.

\end{abstract}

%%%%%%%%%%%%%%%%%%%%%%%%%%%%%%%%%%%%%%%%%%
\section{Introduction}\label{Sec:Intr}
%%%%%%%%%%%%%%%%%%%%%%%%%%%%%%%%%%%%%%%%%%

The description of dynamical systems using  \textit{differential-algebraic equations} (DAEs), which are a combination of differential equations with algebraic constraints, arises in various relevant applications, where the dynamics are algebraically constrained, for instance by tracks, Kirchhoff laws, or conservation laws. To name but a few, DAEs appear naturally in mechanical multibody dynamics~\cite{EichFueh98}, electrical networks~\cite{Riaz08} and chemical engineering~\cite{KumaDaou99}, but also in non-natural scientific contexts such as economics~\cite{LuenArbe77} or demography~\cite{Camp82}. The aforementioned problems often cannot be modeled by ordinary differential equations~(ODEs) and hence it is of practical interest to investigate the properties of DAEs. Due to their power in applications, nowadays DAEs are an established field in applied mathematics and subject of various monographs and textbooks, see e.g.~\cite{BrenCamp89,KunkMehr06,LamoMarz13}.

In the present paper we study state estimation for a class of nonlinear differential-algebraic systems. Nonlinear DAE systems seem to have been first considered by Luenberger~\cite{Luen79a}; cf.\ also the textbooks~\cite{KunkMehr06,LamoMarz13} and the recent works~\cite{Berg16a,Berg17b}. Since it is often not possible to directly measure the state of a system, but only the external signals (input and output) and an internal model are available, it is of interest to construct an ``observing system'' which approximates the original system's state. Applications for observers are for instance error detection and fault diagnosis, disturbance (or unknown input) estimation and feedback control, see e.g.~\cite{CorrCris12,YeuKawa02}.

Several results on observer design for nonlinear DAEs are available in the literature. Lu and Ho~\cite{LuHo06} developed a Luenberger type observer for square systems with Lipschitz continuous nonlinearities, utilising solutions of a certain linear matrix inequality (LMI) to construct the observer. This is more general than the results obtained in~\cite{GaoHo06}, where the regularity of the linear part was assumed. Extensions of the work from~\cite{LuHo06} are discussed in~\cite{ZhanSwai16,ZulfReha16}, e.g.\ for the case of nonlinearities in the output equation. We stress that the approach in~\cite{BoutDaro02} and~\cite{HaTrin04}, where ODE systems with unknown inputs are considered, is similar to the aforementioned since these systems may be treated as DAEs as well. A different approach is taken in~\cite{AsluFris06}, where completely nonlinear DAEs which are semi-explicit and index-1 are investigated, and in~\cite{YangKong13}, where a nonlinear generalized PI observer design is used.

Recently, Gupta et al.~\cite{GuptToma18} presented a reduced-order observer design which is applicable to non-square DAEs with generalized monotone nonlinearities. Systems with nonlinearities which satisfy a more general monotonicity condition are considered in~\cite{YangZhan12}, but the results found there are applicable to square systems only.

A novel observer design using so called innovations has been developed in~\cite{PoldWill98,ValcWill99} and considered for linear DAEs in~\cite{BergReis17c} and for DAEs with Lipschitz continuous nonlinearities in~\cite{Berg19}. Roughly speaking, the innovations are ``[...] a measure for the correctness of the overall internal model at time~$t$''~\cite{BergReis17c}. This approach extends the classical Luenberger type observer design and allows for non-square systems.

It is our aim to present an observer design framework which unifies the above mentioned approaches. To this end, we use the approach from~\cite{BergReis17c} for linear DAEs (which can be non-square) and extend it to incorporate both nonlinearities which are Lipschitz continuous as in~\cite{Berg19,LuHo06} and nonlinearities which are generalized monotone as in~\cite{GuptToma18,YangZhan12}, or combinations thereof. %; the latter is inspired by~\cite{GuptToma18}.
We show that if a certain LMI restricted to a subspace determined by the Wong sequences is solvable, then there exists a state estimator (or observer) for the original system, where the gain matrices corresponding to the innovations in the observer are constructed out of the solution of the LMI. We will distinguish between an \textit{(asymptotic) observer} and a \textit{state estimator}, cf.\ Section~\ref{Definitions}. To this end, we speak of an \textit{observer candidate} before such a system is found to be an observer or a state estimator. We stress that such an observer candidate is a DAE system in general; for the investigation of the existence of ODE observers see e.g.~\cite{Berg19,BergReis19,GuptToma18}.

This paper is organised as follows: We briefly state the basic definitions and some preliminaries on matrix pencils in Section~\ref{Definitions}. The unified framework for the observer design is presented in Section~\ref{System,observer and error dynamics}. In Sections~\ref{Theorems and proofs} and~\ref{Sufficient conditions for asymptotic observers} we state and prove the main results of this paper. Subsequent to the proofs we  give some instructive examples for the theorems in Section~\ref{examples}. A discussion as well as a comparison to the relevant literature is provided in Section~\ref{Comparison with other results} and computational aspects are discussed in Section~\ref{Sec:CompAsp}.

%%%%%%%%%%%%%%%%%%%%%%%%%%%%%%%%%%%%%%%%%%%%%%%%%%%%%%%%%%%%%%%%%%%%%%%%%%%%%%%%%%%%%%%%%%%%%%%%%%%%%%%%%%%%%
\subsection{Nomenclature}\label{Ssec:Nomencl}
%%%%%%%%%%%%%%%%%%%%%%%%%%%%%%%%%%%%%%%%%%%%%%%%%%%%%%%%%%%%%%%%%%%%%%%%%%%%%%%%%%%%%%%%%%%%%%%%%%%%%%%%%%%%%

\begin{tabular}{l l}
$A \in \mathbb{R}^{n \times m}$ & the matrix $A$ is in the set of real $n \times m$ matrices; \\
$\rk A$, $\im A$, $\ker A$ & the rank, image and kernel of $A \in \mathbb{R}^{n \times m}$, resp.; \\
$\cC^k(X \to Y)$ & the set of $k-$times continuously differentiable functions \\
& $f\colon X \to Y$, $k \in \mathbb{N}_0$; \\
$\dom(f)$ & the domain of the function $f$; \\
$A >_V 0$ &$:\iff \forall\, x \in V \setminus{\{0\}}: \,  x^\top A x > 0$, $V \subseteq \mathbb{R}^n$ a subspace; \\
$\mathbb{R}[s]$ & rhe ring of polynomials with coefficients in $\mathbb{R}$.
\end{tabular}

%%%%%%%%%%%%%%%%%%%%%%%%%%%%%%%%%%%%%%%%%%%%%%%%%%%%%%%%%%%%%%%%%%%%%%%%%%%%%%%%%%%%%%%%%%%%%%%%%%%%%%%%%%%%%
\section{Preliminaries  \label{Definitions}}
%%%%%%%%%%%%%%%%%%%%%%%%%%%%%%%%%%%%%%%%%%%%%%%%%%%%%%%%%%%%%%%%%%%%%%%%%%%%%%%%%%%%%%%%%%%%%%%%%%%%%%%%%%%%%

We consider nonlinear DAE systems of the form
\begin{equation}
\label{general}
	\begin{aligned}
	 \ddt E x(t) &= f(x(t), u(t), y(t)) \\
	 y(t) &= h(x(t), u(t)),
	\end{aligned}
\end{equation}
with $E \in \mathbb{R}^{l \times n}$, $f \in \cC(\mathcal{X} \times \mathcal{U} \times \mathcal{Y} \to \mathbb{R}^l)$ and $h \in \cC(\mathcal{X} \times \mathcal{U} \to \mathbb{R}^p)$, where  $\mathcal{X} \subseteq \mathbb{R}^n$, $\mathcal{U} \subseteq \mathbb{R}^m$ and  $\mathcal{Y} \subseteq \mathbb{R}^p$  are open.  The functions $x:I\to\R^n$, $u:I\to\R^m$ and $y:I\to\R^p$ are called the \emph{state}, \emph{input} and \emph{output}  of~\eqref{general}, resp. Since solutions not necessarily exist globally we consider local solutions of~\eqref{general}, which leads to the following solution concept, cf.~\cite{Berg19}.

\begin{Def}
\label{def-solution}
Let $I \subseteq \mathbb{R}$ be an open interval. A trajectory $(x,u,y)\in\cC(I\to\mathcal{X}\times\mathcal{U}\times \mathcal{Y})$ is called \textit{solution} \index{Solution} of~\eqref{general}, if $x \in \cC^1(I \to \mathcal{X})$ and~\eqref{general} holds for all $t \in I$. The set
\[
%\begin{small}
\mathfrak{B}_{\eqref{general}} := \setdef{(x,u,y) \in \cC(I\!\to\!\mathcal{X}\!\times\!\cU\!\times\!\cY) }{ I \subseteq \mathbb{R} \text{ open intvl.,}\ (x,u,y)\, \text{is a solution of~\eqref{general}} }
%\end{small}
\]
of all possible solution trajectories is called the \textit{behavior} \index{Behavior} of system~\eqref{general}.
\end{Def}

We stress that the interval of definition~$I$ of a solution of~\eqref{general} does not need to be maximal and, moreover, it depends on the choice of the input~$u$. Next we introduce the concepts of an acceptor, an (asymptotic) observer and a state estimator. These definitions follow in essence the definitions given in~\cite{Berg19}.

\begin{Def}
\label{def-acceptor}
Consider a system \eqref{general}. The system
\begin{equation}
\label{acceptor}
	\begin{aligned}
\ddt E_o x_o(t) &= f_o( x_o(t), u(t), y(t)),\\
z(t) &= h_o(x_o(t), u(t), y(t)),
	\end{aligned}
\end{equation}
where $E_o \in \mathbb{R}^{l_o \times n_o}$, $f_o \in \cC(\mathcal{X}_o \times \cU \times \cY \to \mathbb{R}^{l_o})$,
$h_o \in \cC(\mathcal{X}_o \times \cU \times \cY \to \mathbb{R}^{p_o}$), $\mathcal{X}_o \subseteq \mathbb{R}^{n_o}$ open, is called \textit{acceptor for \eqref{general}} \index{Acceptor}, if for all $(x,u,y) \in \mathfrak{B}_{\eqref{general}}$ with $I=\dom(x)$, there exist $x_o \in \cC^1(I \to \mathcal{X}_o)$, $z\in \cC(I \to \mathbb{R}^{p_o})$ such that
\begin{equation*}
\left(x_o, \begin{smallpmatrix} u\\ y\end{smallpmatrix}, z \right) \in \mathfrak{B}_{\eqref{acceptor}}.
\end{equation*}
\end{Def}

The definition of an acceptor shows that the original system influences, or may influence, the acceptor but not vice-versa, i.e., there is a directed signal flow from~\eqref{general} to~\eqref{acceptor}, see Fig.~\ref{Fig:acceptor}.

\begin{figure}[h!tb]
\begin{center}
%\resizebox{9cm}{!}{
\begin{tikzpicture}[very thick,node distance = 10ex, box/.style={fill=white,rectangle, draw=black}, blackdot/.style={inner sep = 0, minimum size=3pt,shape=circle,fill,draw=black},plus/.style={fill=white,circle,inner
sep = 0,thick,draw},metabox/.style={inner sep = 0ex,rectangle,draw,dotted,fill=gray!20!white}]
  %\begin{tikzpicture}[very thick,node distance = 9ex, box/.style={fill=white,rectangle, draw=black}, blackdot/.style={inner sep = 0, minimum size=3pt,shape=circle,fill,draw=black},blackdotsmall/.style={inner sep = 0, minimum size=0.1pt,shape=circle,fill,draw=black},plus/.style={fill=white,circle,inner sep = 0,very thick,draw},metabox/.style={inner sep = 3ex,rectangle,draw,dotted,fill=gray!20!white}]
  \tikzset{>=latex}

    \node (EA)     [box,minimum size=6ex]  {$\begin{aligned} E\, \dot x(t) &= f\big(x(t),u(t),y(t)\big)\\
      y(t) &= h\big(x(t),u(t)\big)\end{aligned}$};
    \node (Acc)     [box,minimum size=6ex,below of = EA, yshift = -10ex]  {$\begin{aligned} E_o\, \dot x_o(t) &= f_o\big(x_o(t),u(t),y(t)\big)\\
      z(t) &= h_o\big(x_o(t),u(t),y(t)\big)\end{aligned}$};
    \node (rightofAcc)   [right of = Acc, xshift = 20ex] {};
    \node (1leftofAcc)   [left of = Acc, xshift = -3.3ex, yshift = 2ex] {};
    \node (2leftofAcc)   [left of = Acc, xshift = -3.3ex, yshift = -2ex] {};

    \node (leftofEA)   [blackdot,left of = EA, xshift=-20ex] {};
    \node (knick1)   [below of = leftofEA,,minimum size=0ex, xshift=10ex] {};
    \node (lleftofEA)   [left of = leftofEA] {};
    \node (rightofEA)   [blackdot,right of = EA, xshift=10ex] {};
    \node (rrightofEA)   [right of = rightofEA] {};

    \draw[-] (lleftofEA) -- (EA) node[midway,above] {$u(t)$};
    \draw[-] (EA) -- (rrightofEA) node[midway,above] {$y(t)$};
    \draw[->] (Acc) -- (rightofAcc) node[midway,above] {$z(t)$};
    \draw[-] (rightofEA) |- (knick1.west) node[midway,above] {};
    \draw[->] (leftofEA) |- (2leftofAcc) node[midway,above] {};
    \draw[->] (knick1.west) |- (1leftofAcc) node[midway,above] {};

  \end{tikzpicture}
%}
\end{center}
           % \vspace{-3mm}
  \caption{Interconnection with an acceptor} \label{Fig:acceptor}
\end{figure}
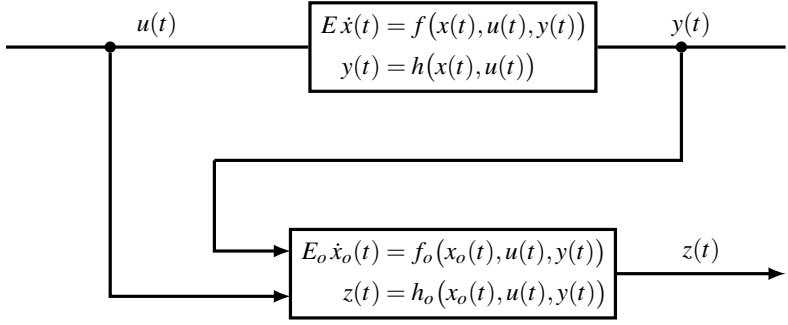

\begin{Def}
\label{def-observer}
Consider a system \eqref{general}. Then a system~\eqref{acceptor} with $p_o=n$ is called
\begin{enumerate}[a)]
\item an {\em observer for~\eqref{general}}\index{Observer}, if it is an acceptor for~\eqref{general}, and
\begin{equation}\label{eq-observer}
\begin{aligned}
&\forall\, I\subseteq\R\ \text{open interval}\ \forall\, t_0\in I\ \forall\, (x,u,y,x_o,z)\in\cC(I\to\mathcal{X} \times \mathcal{U} \times \mathcal{Y} \times \mathcal{X}_o  \times \mathbb{R}^{n}):\\
& \Big((x,u,y)\in\mathfrak{B}_{\eqref{general}}\;\wedge\; (x_o,\left(\begin{smallmatrix}u\\y\end{smallmatrix}\right),z)\in\mathfrak{B}_{\eqref{acceptor}}\ \wedge\;
Ez(t_0)=Ex(t_0)\Big)\ \ \Longrightarrow\ \ z = x;
\end{aligned}
\end{equation}
\item a {\em state estimator for~\eqref{general}}\index{State estimator}, if it is an acceptor for~\eqref{general}, and
\begin{equation}\label{estimator}
\begin{aligned}
&\forall\, t_0\in \R\ \forall\, (x,u,y,x_o,z)\in\cC([t_0,\infty)\to\mathcal{X} \times \mathcal{U} \times \mathcal{Y} \times \mathcal{X}_o  \times \mathbb{R}^{n}):\\
&\Big( (x,u,y)\in\mathfrak{B}_{\eqref{general}}\ \wedge\  (x_o,\left(\begin{smallmatrix}u\\y\end{smallmatrix}\right),z)\in\mathfrak{B}_{\eqref{acceptor}}\Big)\ \ \Longrightarrow\ \ \lim_{t\to\infty} z(t)-x(t) = 0;
\end{aligned}
\end{equation}
\item an {\em asymptotic observer for~\eqref{general}}, if it is an observer and a state estimator for~\eqref{general}.
\end{enumerate}
\end{Def}

The property of being a state estimator is much weaker than being an asymptotic observer. Since there is no requirement such as~\eqref{eq-observer} it might even happen that the state estimator's state matches the original system's state for some time, but eventually evolves in a different direction.

Concluding this section we recall some important concepts for matrix pencils. First, a matrix pencil $sE-A \in \mathbb{R}[s]^{l \times n}$ is called \textit{regular}, if $l=n$ and $\det(sE-A) \neq 0 \in \mathbb{R}[s]$. An important geometric tool are the \textit{Wong sequences}, named after Wong~\cite{Wong74}, who was the first to use both sequences for the analysis of matrix pencils. The Wong sequences are investigated and utilized for the decomposition of matrix pencils in~\cite{BergIlch12a,BergTren12,BergTren13}.

\begin{Def}
\label{Wong-sequences}
Consider a matrix pencil $sE-A \in \mathbb{R}[s]^{l \times n}$. The \textit{Wong sequences}\index{Wong sequence} are sequences of subspaces, defined by\\[2mm]
\begin{tabular}{l l l}
$\mathcal{V}_{[E,A]}^0 := \mathbb{R}^n$, & $\mathcal{V}_{[E,A]}^{i+1} := A^{-1}(E \mathcal{V}_{[E,A]}^i) \subseteq \mathbb{R}^n$, & $\mathcal{V}_{[E,A]}^* := \bigcap\limits_{i \in  \mathbb{N}_0}\mathcal{V}_{[E,A]}^i$,\\
$\mathcal{W}_{[E,A]}^0 := \{0\}$, & $\mathcal{W}_{[E,A]}^{i+1} := E^{-1}(A\mathcal{W}_{[E,A]}^i) \subseteq \mathbb{R}^n$, &
 $\mathcal{W}_{[E,A]}^* := \bigcup\limits_{i \in  \mathbb{N}_0}\mathcal{W}_{[E,A]}^i$,
\end{tabular}

where $A^{-1}(S) = \{ x \in \mathbb{R}^n \mid Ax \in S \}$ is the preimage of $S \subseteq \mathbb{R}^l$ under $A$. The subspaces~$\mathcal{V}_{[E,A]}^*$ and~$\mathcal{W}_{[E,A]}^*$ are called the \textit{Wong limits}\index{Wong limit}.
\end{Def}

As shown in~\cite{BergTren12} the Wong sequences terminate, are nested and satisfy
\begin{equation*}
	\begin{aligned}
&\exists\, k^* \in \mathbb{N}\,\, \forall j \in \mathbb{N} : \mathcal{V}_{[E,A]}^0 \supsetneq \mathcal{V}_{[E,A]}^1 \supsetneq \dots \supsetneq \mathcal{V}_{[E,A]}^{k^*} = \mathcal{V}_{[E,A]}^{k^*+j} = \mathcal{V}_{[E,A]}^{*} \supseteq \ker(A), \\
&\exists\, l^* \in \mathbb{N} \,\, \forall j \in \mathbb{N} : \mathcal{W}_{[E,A]}^0 \subseteq \ker(E) = \mathcal{W}_{[E,A]}^1 \subsetneq \dots \subsetneq \mathcal{W}_{[E,A]}^{l^*} = \mathcal{W}_{[E,A]}^{l^*+j} = \mathcal{W}_{[E,A]}^{*}.
	\end{aligned}
\end{equation*}

\begin{Rem}\label{Rem-B-trajectories-in-V}
Let $sE-A \in \mathbb{R}[s]^{l \times n}$ and consider the associated DAE $\ddt E x(t) = Ax(t)$. In view of Definition~\ref{def-solution} we may associate with the previous equation the behavior
\[
    \mathfrak{B}_{[E,A]} = \setdef{ x \in \cC^1(I\to \mathbb{R}^n) }{ E \dot{x} = Ax,\ I\subseteq \mathbb{R} \text{ open interval}}.
\]
We have that all trajectories in $\mathfrak{B}_{[E,A]}$ evolve in $\mathcal{V}_{[E,A]}^{*}$, that is
\begin{equation}
\label{B-trajectories-in-V}
\forall\, x \in \mathfrak{B}_{[E,A]} \ \ \forall\, t \in \dom(x): \ \ x(t) \in \mathcal{V}_{[E,A]}^{*}.
\end{equation}
This can be seen as follows: For $x \in \mathfrak{B}_{[E,A]}$ we have that $x(t) \in \mathbb{R}^n = \mathcal{V}_{[E,A]}^0$ for all $t \in \dom(x)$. Since the linear spaces $\mathcal{V}_{[E,A]}^i$ are closed they are invariant under differentiation and hence $\dot{x}(t) \in \mathcal{V}_{[E,A]}^0$. Due to the fact that $x \in \mathfrak{B}_{[E,A]}$ it follows for all $t \in \dom(x)$ that $x(t) \in A^{-1} (E\mathcal{V}_{[E,A]}^0) = \mathcal{V}_{[E,A]}^1$.
Now assume that $x(t) \in \mathcal{V}_{[E,A]}^i$ for some $i \in \mathbb{N}_0$ and all $t\in\dom(x)$. By the previous arguments we find that $x(t) \in A^{-1}(E\mathcal{V}_{[E,A]}^i) = \mathcal{V}_{[E,A]}^{i+1}$.
\end{Rem}

An important concept in the context of DAEs is the index of a matrix pencil, which is based on the (quasi-)Weierstra\ss\ form (QWF), cf.~\cite{BergIlch12a,Gant59d,KunkMehr06,LamoMarz13}.

\begin{Def}
Consider a regular matrix pencil $sE-A \in \mathbb{R}[s]^{n \times n}$ and let $S,T\in\R^{n\times n}$ be invertible such that
\[ S(sE-A)T =
s \left[ \begin{matrix}
I_r & 0 \\ 0 & N
\end{matrix} \right]
-
\left[ \begin{matrix}
J & 0 \\ 0 & I_{n-r}
\end{matrix} \right] \]
for some $J\in\R^{r\times r}$ and nilpotent $N\in\R^{(n-r)\times (n-r)}$. Then
\[
    \nu := \begin{cases}
      0, & \text{if $r=n$},\\
      \min\setdef{k\in\N}{N^k=0}, & \text{if $r<n$}
    \end{cases}
\]
is called the \textit{index}\index{Index} of $sE-A$.
\end{Def}

The index is independent of the choice of~$S,T$ and can be computed via the Wong sequences as shown in~\cite{BergIlch12a}.

%%%%
\section{System, observer candidate and error dynamics} \label{System,observer and error dynamics}
%%%%

In this section we present the observer design used in this paper, which invokes so called innovations and was developed in~\cite{PoldWill98,ValcWill99} for linear behavioral systems. It is an extension of the classical approach to observer design which goes back to Luenberger, see~\cite{Luen64,Luen71}.

We consider nonlinear DAE systems of the form
\begin{equation}
\label{DAE}
	\begin{aligned}
\ddt E x(t) &= Ax(t) + B_L f_L(x(t),u(t),y(t)) + B_M f_M(Jx(t),u(t),y(t)),\\
y(t) &= Cx(t) + h(u(t)),
	\end{aligned}
\end{equation}
where $E,A \in \mathbb{R}^{l \times n}$ with $0\leq r=\rk(E)\leq n$, $B_L \in \mathbb{R}^{l \times q_L}$, $B_M \in \mathbb{R}^{l \times q_M}$, $J \in \mathbb{R}^{q_M \times n}$ with $\rk J =~q_M$, $C \in \mathbb{R}^{p \times n}$ and $h \in \cC(\mathcal{U} \to \mathbb{R}^p)$ with $\mathcal{U} \subseteq \mathbb{R}^m$ open. Furthermore, for some open sets $\mathcal{X} \subseteq \mathbb{R}^n, \mathcal{Y} \subseteq \mathbb{R}^p$ and $\hat \cX := J \cX \subseteq\R^{q_M}$, the nonlinear function~$f_L:\mathcal{X} \times \mathcal{U} \times \mathcal{Y} \to \mathbb{R}^{q_{L}}$ satisfies a Lipschitz condition\index{Lipschitz condition} in the first variable
\begin{equation}
\label{Lipschitz}
\forall\, x,z \in \mathcal{X}\ \forall\, u \in \mathcal{U}\ \forall\, y \in \mathcal{Y}: \ \ \|f_L(z,u,y)-f_L(x,u,y)\| \leq \|F(z-x)\|
\end{equation}
with $F \in \mathbb{R}^{j \times n}$, $j \in \mathbb{N}$; and $f_M: \hat \cX \times \mathcal{U} \times \mathcal{Y} \to \mathbb{R}^{q_{M}}$ satisfies a generalized monotonicity\index{Generalized monotonicity} condition in the first variable
\begin{equation}
\label{Monotonicity}
\forall\, x,z \in \hat \cX\ \forall\, u \in \mathcal{U}\ \forall\, y \in \mathcal{Y}: \ \ (z-x)^\top \Theta \big( f_M(z,u,y)-f_M(x,u,y) \big) \geq \frac{1}{2}\mu \|z-x\|^2
\end{equation}
for some $\Theta \in \mathbb{R}^{q_M \times q_M}$ and $\mu \in \mathbb{R}$. We stress that $\mu<0$ is explicitly allowed and~$\Theta$ can be singular, i.e., in particular~$\Theta$ does not necessarily satisfy any definiteness conditions as in~\cite{YangZhan12}. We set $B := [B_L, B_M] \in \mathbb{R}^{l \times (q_L + q_M)}$ and
\[
    f : \mathcal{X} \times \mathcal{U} \times \mathcal{Y} \to \mathbb{R}^{q_{L}} \times \mathbb{R}^{q_M},\ (x,u,y) \mapsto \begin{pmatrix} f_L(x,u,y) \\ f_M(Jx,u,y) \end{pmatrix}.
\]

Let us  consider a system~\eqref{DAE} and assume that $n=l$. Then another system driven by the external variables~$u$ and~$y$ of~\eqref{DAE} of the form
\begin{equation}
\label{Lu-observer}
	\begin{aligned}
\ddt E z(t) &= Az(t) + Bf(z(t),u(t),y(t)) + L(y(t)-\hat{y}(t)) \\
&= Az(t) + Bf(z(t),u(t),y(t)) + L(Cx(t)-Cz(t)) \\
&= (A-LC)z(t) + Bf(z(t),u(t),y(t)) + L \underbrace{Cx(t)}_{= y(t) - h(u(t))} \\
\text{with} \ \	\hat{y}(t) &= C z(t) + h(u(t))
		\end{aligned}
\end{equation}
is a Luenberger type observer\index{Luenberger type observer}, where $L\in\R^{n\times p}$ is the observer gain. The dynamics for the error state $e(t) = z(t)-x(t)$ read
\begin{equation*}
\ddt E e(t) = (A-LC)e(t) + B\big(f(x(t),u(t),y(t))-f(z(t),u(t),y(t))\big).
\end{equation*}
The observer \eqref{Lu-observer} incorporates a copy of the original system, and in addition the outputs' difference $\hat{y}(t)-y(t)$, the influence of which is weighted with the observer gain~$L$.

In this paper we consider a generalization of the design~\eqref{Lu-observer} which incorporates an extra variable~$d$ that takes the role of the innovations\index{Innovations}. The innovations are used to describe ``the difference between what we actually observe and what we had expected to observe''~\cite{PoldWill98}, and hence they generalize the effect of the observer gain~$L$ in~\eqref{Lu-observer}. We consider the following observer candidate, which is an additive composition of an internal model of the system~\eqref{DAE} and a further term which involves the innovations:
\begin{equation}
\label{eq:state-estimator}
	\begin{aligned}	
\ddt Ez(t)&=Az(t)+Bf(z(t),u(t),y(t))+L_1 d(t) \\
0&=Cz(t)-y(t)+h(u(t))+L_2 d(t),
	\end{aligned}
\end{equation}
where $x_o(t) = \begin{smallpmatrix}z(t)\\ d(t)\end{smallpmatrix}$ is the observer state and $L_1 \in \mathbb{R}^{l \times k}$, $L_2 \in \mathbb{R}^{p \times k}$, $\mathcal{X}_o = \mathcal{X} \times \mathbb{R}^k$. From the second line in~\eqref{eq:state-estimator} we see that the innovations term balances the difference between the system's and the observer's output. In a sense, the smaller the variable~$d$, the better the approximate state~$z$ in~\eqref{eq:state-estimator} matches the state~$x$ of the original system~\eqref{DAE}.

We stress that $n \neq l$ is possible in general, and if~$L_2$ is invertible, then the observer candidate reduces to
\begin{equation}
\label{Luen-from-inno}
	\begin{aligned}
\ddt E {z}(t) &= Az(t) + Bf(z(t),u(t),y(t))+L_1 L_2^{-1}(y(t) - Cz(t) - h(u(t))) \\
&= (A - L_1 L_2^{-1} C)z(t) + Bf(z(t),u(t),y(t)) + L_1 L_2^{-1}(\underbrace{y(t) - h(u(t))}_{=Cx(t)}),
	\end{aligned}
\end{equation}
which is a Luenberger type observer of the form \eqref{Lu-observer} with gain $L=L_1 L_2^{-1}$. Hence the Luenberger type observer is a special case of the observer design~\eqref{eq:state-estimator}. Being square is a necessary condition for invertibility of~$L_2$, i.e., $k=p$.

For later use we consider the dynamics of the error state $e(t):= z(t) - x(t)$ between systems~\eqref{DAE} and~\eqref{eq:state-estimator},
\begin{equation}
\label{ed}
\begin{aligned}	
\ddt Ee(t) &=Ae(t)+B\phi(t) +L_1d(t) \\
0&=Ce(t)+L_2d(t),
\end{aligned}
\end{equation}
where
\[
    \phi (t) := f(z(t),u(t),y(t))-f(x(t),u(t),y(t)) = \begin{smallpmatrix} f_L(z(t),u(t),y(t)) - f_L(x(t),u(t),y(t)) \\ f_M(Jz(t),u(t),y(t)) - f_M(Jx(t),u(t),y(t)) \end{smallpmatrix},
\]
and rewrite \eqref{ed} as
\begin{equation}
\label{ED}
\ddt \mathcal{E}\begin{smallpmatrix} e(t)\\ d(t)\end{smallpmatrix} = \mathcal{A} \begin{smallpmatrix} e(t)\\ d(t)\end{smallpmatrix}  +\mathcal{B}\phi(t),
\end{equation}
where
\begin{align*}
   & \mathcal{E} = \begin{bmatrix} E & 0 \\ 0 & 0 \end{bmatrix} \in \mathbb{R}^{(l+p) \times (n+k)},\quad
\mathcal{A} = \begin{bmatrix} A & L_1 \\ C & L_2 \end{bmatrix} \in \mathbb{R}^{(l+p) \times (n+k)}\\
and\quad & \mathcal{B} =  \begin{bmatrix} B \\ 0 \end{bmatrix} \in \mathbb{R}^{(l+p) \times (q_L + q_M)}.
\end{align*}
The following lemma is a consequence of~\eqref{B-trajectories-in-V}.

\begin{Lem}\label{Lem-augmented-trajectories-in-V}
Consider a system \eqref{DAE} and the observer candidate \eqref{eq:state-estimator}. Then \eqref{eq:state-estimator} is an acceptor for~\eqref{DAE}. Furthermore, for all open intervals $I \subseteq \mathbb{R}$, all $(x,u,y) \in \mathfrak{B}_{\eqref{DAE}}$ and all $\left(\binom{z}{d},\binom{u}{y},z \right) \in \mathfrak{B}_{\eqref{eq:state-estimator}}$ with $\dom(x)=\dom\binom{z}{d} = I$ we have:
\begin{equation}
\label{augmented-trajectories-in-V}
\forall\, t \in I: \ \left(\begin{matrix} e(t) \\ d(t) \\ \phi(t) \end{matrix}\right) \in \mathcal{V}^*_{\left[[\mathcal{E}, 0], [\mathcal{A}, \mathcal{B}] \right]} .
\end{equation}
\end{Lem}
\begin{proof}
Let $I \subseteq \mathbb{R}$ be an open interval and $(x,u,y) \in \mathfrak{B}_{\eqref{DAE}}$. For any $(x,u,y) \in \mathfrak{B}_{\eqref{DAE}}$ it holds $\left(\binom{x}{0},\binom{u}{y},x \right) \in \mathfrak{B}_{\eqref{eq:state-estimator}}$, hence \eqref{eq:state-estimator} is an acceptor for \eqref{DAE}.\\
Now let $(x,u,y) \in \mathfrak{B}_{\eqref{DAE}}$ and $\left(\binom{z}{d},\binom{u}{y},z \right) \in \mathfrak{B}_{\eqref{eq:state-estimator}}$, with $I=\dom(x)=\dom\binom{z}{d}$ and rewrite~\eqref{ED} as
\begin{equation*}
\label{ED-augmented}
\ddt [\mathcal{E}, 0]\left(\begin{matrix} e(t) \\ d(t) \\ \phi(t) \end{matrix}\right) = [\mathcal{A}, \mathcal{B}] \left(\begin{matrix} e(t) \\ d(t) \\ \phi(t) \end{matrix}\right).
\end{equation*}
Then~\eqref{augmented-trajectories-in-V} is immediate from Remark~\ref{Rem-B-trajectories-in-V}.
\end{proof}

In the following lemma we show that for a state estimator to exist, it is necessary that the system~\eqref{DAE} does not contain free state variables, i.e., solutions (if they exist) are unique.

\begin{Lem}\label{Lem:rank-E'-A'} Consider a system \eqref{DAE} and the observer candidate \eqref{eq:state-estimator}. If~\eqref{eq:state-estimator} is a state estimator for~\eqref{DAE}, then either
\begin{equation}\label{eq:sln-dont-exist}
\begin{aligned}
    & \Big(\forall\,(x,u,y)~\in~\mathfrak{B}_{\eqref{DAE}}\ \exists\, t_0\in\R:\ \dom(x) \cap [t_0,\infty) = \emptyset \Big)\\
    \vee\quad & \Big(\forall\,\big(\begin{smallpmatrix}  z \\ d \end{smallpmatrix}, \begin{smallpmatrix}  u \\ y \end{smallpmatrix}, z \big)~\in~\mathfrak{B}_{\eqref{eq:state-estimator}}\ \exists\, t_0\in\R:\ \dom(z,d) \cap [t_0,\infty) = \emptyset \Big),
\end{aligned}
\end{equation}
or we have $\rk_{\R(s)} \begin{smallbmatrix} sE -A \\  C \end{smallbmatrix}~=~n$.
\end{Lem}
\begin{proof}
Let~\eqref{eq:state-estimator} be a state estimator for~\eqref{DAE} and assume that~\eqref{eq:sln-dont-exist} is not true. Set $E' :=\begin{smallbmatrix} E \\ 0 \end{smallbmatrix}$, $A' := \begin{smallbmatrix} A \\ C \end{smallbmatrix}$ and let $(x,u,y)~\in~\mathfrak{B}_{\eqref{DAE}}$ with $[t_0,\infty)\subseteq \dom(x)$ for some $t_0\in\R$. Then we have that, for all $t\ge t_0$,
\begin{equation}\label{eq:DAE-E'-A'}
\ddt E' x(t) = \begin{bmatrix} A \\ C \end{bmatrix} x(t) + \begin{bmatrix} B & L_1 \\ 0 & L_2 \end{bmatrix} \begin{pmatrix} f(x(t),u(t),y(t)) \\ d(t) \end{pmatrix} =: A' x(t) + g(x(t),u(t),y(t),d(t))
\end{equation}
with $d(t) \equiv 0$. Using \cite[Thm. 2.6]{BergTren12} we find matrices $S \in Gl_{l+p}(\R)$, $T \in Gl_{n}(\R)$ such that
\begin{equation}
\label{Trafo:QKF}
S\left(sE' - A'\right)T =
s\begin{bmatrix} E_P & 0 & 0 \\ 0& E_R & 0 \\ 0& 0&E_Q \end{bmatrix} - \begin{bmatrix} A_P & 0 & 0 \\ 0& A_R & 0 \\ 0 & 0 & A_Q \end{bmatrix},
\end{equation}
where
\begin{enumerate}
\item[(i)] $E_P , A_P \in \R^{m_P \times n_P}, m_P < n_P$, are such that $\rk_{\C}(\lambda E_P-A_P) = m_P$ for all $\lambda \in \C \cup \{ \infty \}$,
\item[(ii)] $E_R , A_R \in \R^{m_R \times n_R}, m_R = n_R$, with $sE_R - A_R$ regular,
\item[(iii)] $E_Q , A_Q \in \R^{m_Q \times n_Q}, m_Q > n_Q$, are such that $\rk_{\C}(\lambda E_Q-A_Q) = n_Q$ for all $\lambda \in \C \cup \{ \infty \}$.
\end{enumerate}
We consider the underdetermined pencil $sE_P - A_P$ in \eqref{Trafo:QKF} and the corresponding DAE. If $n_P = 0$, then~\cite[Lem.~3.1]{BergTren12} implies that $\rk_{\R(s)} sE_Q-A_Q = n_Q$ and invoking $\rk_{\R(s)} s E_R- A_R = n_R$ gives that $\rk_{\R(s)} s E' -A' = n$. So assume that $n_p>0$ in the following and set
\begin{equation*}
\begin{pmatrix} x_p \\ x_R \\ x_Q \end{pmatrix} := T^{-1} x, \quad
\begin{pmatrix} g_p \\ g_R \\ g_Q \end{pmatrix} := S g.
\end{equation*}
If $m_p = 0$, then $x_P$ can be chosen arbitrarily. Otherwise, we have
\begin{equation}
\label{eq:sEP-AP}
\ddt E_P x_P(t) = A_P x_P(t) +
g_P \left(T \begin{smallpmatrix} x_P(t) \\ x_R(t) \\ x_Q(t) \end{smallpmatrix},u(t),y(t),d(t)\right).
\end{equation}
As a consequence of~\cite[Lem.~4.12]{BergTren12} we may w.l.o.g.\ assume that $sE_P-A_P = s[I_{m_p},0]-[N,R]$ with $R\in\R^{m_P\times (n_P-m_P)}$ and nilpotent $N\in\R^{m_P\times m_P}$. Partition $x_P = \begin{smallpmatrix} x_P^1 \\ x_P^2 \end{smallpmatrix}$, then~\eqref{eq:sEP-AP} is equivalent to
\begin{align}
\dot{x}_P^1(t) = N x_P^1(t) + R x_P^2(t) + g_P\big( T(x_P^1(t)^\top, x_P^2(t)^\top, x_R(t)^\top, x_Q(t)^\top)^\top,u(t),y(t),d(t)\big)
\end{align}
for all $t\ge t_0$, and hence $x_P^2\in\cC([t_0,\infty)\to\R^{n_P-m_P})$ can be chosen arbitrarily and every choice preserves $[t_0,\infty)\subseteq \dom(x)$. Similarly, if $\big(\begin{smallpmatrix}  z \\ d \end{smallpmatrix}, \begin{smallpmatrix}  u \\ y \end{smallpmatrix}, z \big)~\in~\mathfrak{B}_{\eqref{eq:state-estimator}}$ with $[t_0,\infty)\subseteq \dom(z)$ -- w.l.o.g.\ the same~$t_0$ can be chosen -- then~\eqref{eq:DAE-E'-A'} is satisfied for $x=z$ and, proceeding in an analogous way, $z_P^2$ can be chosen arbitrarily, in particular such that $\lim_{t \to \infty}~z_P^2(t)~\neq~\lim_{t \to \infty}~x_P^2(t)$. Therefore, $\lim_{t \to \infty}~z(t)~-~x(t)~=~\lim_{t \to \infty}~e(t)~\neq~0$, which contradicts that~\eqref{eq:state-estimator} is a state estimator for~\eqref{DAE}. Thus $n_P=0$ and $\rk_{\R(s)} s E' -A' = n$ follows.
\end{proof}

As a consequence of Lemma~\ref{Lem:rank-E'-A'}, a necessary condition for~\eqref{eq:state-estimator} to be a state estimator for~\eqref{DAE} is that $n \leq l+p$. This will serve as a standing assumption in the subsequent sections.

%%%
\section{Sufficient conditions for state estimators} \label{Theorems and proofs}
%%%

In this section we show that if certain matrix inequalities\index{Matrix inequality} are satisfied, then there exists a state estimator for system~\eqref{DAE} which is of the form~\eqref{eq:state-estimator}. The design parameters of the latter can be obtained from a solution of the matrix inequalities. The proofs of the subsequent theorems are inspired by the work of Lu and Ho~\cite{LuHo06} and by~\cite{Berg19}, where LMIs are considered on the Wong limits only.

\begin{Thm} \label{Version-1}
Consider a system \eqref{DAE} with $n\le l+p$ which satisfies conditions \eqref{Lipschitz} and \eqref{Monotonicity}. Let $k \in \N_0$ and denote with $\mathcal{V}^*_{\left[[\mathcal{E}, 0], [\mathcal{A}, \mathcal{B}] \right]}$ the Wong limit\index{Wong limit} of the pencil $s[\mathcal{E},0]-[\mathcal{A},\mathcal{B}] \in \mathbb{R}[s]^{(l+p)\times(n+k+q_L+q_M)}$, and $\overline{\mathcal{V}}^*_{\left[[\mathcal{E}, 0], [\mathcal{A}, \mathcal{B}] \right]}~:=~\left[I_{n+k}, 0 \right] \mathcal{V}^*_{\left[[\mathcal{E}, 0], [\mathcal{A}, \mathcal{B}] \right]}$. Further let
$\hat{A}~=~\left[ \begin{smallmatrix} A & 0 \\ C & 0 \end{smallmatrix} \right]$,
$H=\left[ \begin{smallmatrix} 0_{n\times n} & 0 \\ 0 & I_k \end{smallmatrix} \right]~=~H^\top$,
$\mathcal{F}~=~[ F , 0]~\in \mathbb{R}^{j \times (n+k)}$, $j~\in~\mathbb{N}$,
$\hat{\Theta}~=~\begin{bmatrix}0& J^\top \Theta \\0& 0 \end{bmatrix}~\in~\mathbb{R}^{(n+k) \times (q_L+q_M)}$,
$\cJ~=~\begin{bmatrix} J^\top J & 0 \\0 & 0 \end{bmatrix}~\in~\mathbb{R}^{(n+k) \times (n+k)}$ and
$\Lambda_{q_L}~:=~\begin{bmatrix} I_{q_L} & 0 \\0 & 0 \end{bmatrix}~\in~\mathbb{R}^{(q_L+q_M) \times (q_L+q_M)}$.
	
If there exist $\delta > 0$, $\mathcal{P}~\in~\mathbb{R}^{(l+p) \times (n+k)}$ and  $\mathcal{K}~\in~\mathbb{R}^{(n+k) \times (n+k)}$ such that
\begin{equation}
\label{eq:Q-MI}
\mathcal{Q} :=
\begin{bmatrix}
\hat{A}^\top\mathcal{P} + \mathcal{P}^\top \hat{A} + H^\top \mathcal{K}^\top + \mathcal{K}H+ \delta \mathcal{F}^\top \mathcal{F} -\mu \cJ & \mathcal{P}^\top\mathcal{B}+\hat{\Theta} \\
\mathcal{B}^\top\mathcal{P}+\hat{\Theta}^\top & -\delta \Lambda_{q_L}
\end{bmatrix} <_{\mathcal{V}^*_{\left[[\mathcal{E}, 0], [\mathcal{A}, \mathcal{B}] \right]}} 0
\end{equation}
and
\begin{equation}
\label{eq:EP-MI}
\mathcal{P}^\top \mathcal{E} = \mathcal{E}^\top \mathcal{P} >_{\overline{\mathcal{V}}^*_{\left[[\mathcal{E}, 0], [\mathcal{A}, \mathcal{B}] \right]}} 0,
\end{equation}
then for all $L_1 \in \mathbb{R}^{l\times k}$, $L_2 \in \mathbb{R}^{p \times k}$ such that $\cP^\top \begin{smallbmatrix} 0 & L_1 \\ 0 & L_2 \end{smallbmatrix}~=~\mathcal{K} H$ the  system~\eqref{eq:state-estimator} is a state estimator\index{State estimator} for~\eqref{DAE}.

Furthermore, there exists at least one such pair $L_1, L_2$ if, and only if, $\im \cK H \subseteq \im \cP^\top$.
\end{Thm}

\begin{proof}
Using Lemma~\ref{Lem-augmented-trajectories-in-V}, we have that \eqref{eq:state-estimator} is an acceptor\index{Acceptor} for \eqref{DAE}. To show that \eqref{eq:state-estimator} satisfies condition \eqref{estimator} let $t_0 \in \mathbb{R}$ and $(x,u,y,x_o,z) \in \cC([t_0,\infty) \to \mathcal{X} \times \mathcal{U} \times \mathcal{Y} \times \mathcal{X}_o \times \mathbb{R}^n)$ such that $(x,u,y) \in \mathfrak{B}_{\eqref{DAE}}$ and $(x_o,\binom{u}{y},z) \in \mathfrak{B}_{\eqref{eq:state-estimator}}$, with $x_o(t) = \binom{z(t)}{d(t)}$ and $\mathcal{X}_o = \mathcal{X} \times \mathbb{R}^k$.

The last statement of the theorem is clear. Let $\hat{L} = [0_{(l+p)\times n}, \ast]$ be a solution of $\cP^\top \hat{L}~=~\mathcal{K} H$ and $\mathcal{A}~=~\hat{A}~+~\hat{L}$, further set $\eta(t) := \begin{smallpmatrix} e(t)\\ d(t) \end{smallpmatrix}$, where $e(t)~=~z(t)~-~x(t)$. Recall that
\begin{align*}
    \phi(t) &= f(z(t),u(t),y(t))-f(x(t),u(t),y(t)) \\
    &= \begin{pmatrix} f_L(z(t),u(t),y(t))-f_L(x(t),u(t),y(t))\\ f_M(Jz(t),u(t),y(t))-f_M(Jx(t),u(t),y(t))\end{pmatrix} =: \begin{pmatrix} \phi_L(t)\\ \phi_M(t)\end{pmatrix}.
\end{align*}
In view of condition~\eqref{Lipschitz} we have for all $t\geq t_0$ that
\begin{equation}
\label{Lipschitz-geq-0}
\delta (\eta^\top(t) \mathcal{F}^\top\mathcal{F} \eta(t) - \phi_L^\top(t) \phi_L(t)) \geq 0
\end{equation}
and by \eqref{Monotonicity}
\begin{equation}
\label{Monotonicity-geq-0}
([J,0]\eta(t))^\top \Theta \phi_M(t) + \phi_M^\top (t) \Theta^\top  [J,0]\eta(t) - \mu ([J,0]\eta(t))^\top ([J,0]\eta(t)) \geq 0.
\end{equation}
Now assume that \eqref{eq:Q-MI} and \eqref{eq:EP-MI} hold. Consider a Lyapunov function\index{Lyapunov function} candidate
\[
\tilde{V} \colon \mathbb{R}^{n+k} \to \mathbb{R}, \ \ \eta \mapsto \eta^\top \mathcal{E}^\top \mathcal{P} \eta
\]
and calculate the derivative along solutions for $t \geq t_0$:
\begin{align}
\ddt\tilde{V}(\eta(t)) \notag
&= \dot{\eta}^\top(t) \mathcal{E}^\top \mathcal{P} \eta(t) + \eta^\top(t) \mathcal{P}^\top \mathcal{E} \dot{\eta}(t)  \\ \notag
&=\left(\mathcal{A}\eta(t) + \mathcal{B}\phi(t) \right)^\top \mathcal{P} \eta(t) + \eta^\top(t) \mathcal{P}^\top \left(\mathcal{A}\eta(t) + \mathcal{B}\phi(t) \right) \\ \notag
&= \eta^\top(t) \mathcal{A}^\top\mathcal{P}\eta(t) + \eta^\top(t) \mathcal{P}^\top\mathcal{A}\eta(t) + \phi^\top(t) \mathcal{B}^\top\mathcal{P}\eta(t) + \eta^\top(t) \mathcal{P}^\top\mathcal{B}\phi(t) \\ \notag
%&= \eta^\top(t) (\hat{A} + \hat{L})^\top \mathcal{P} \eta(t) + \eta^\top(t) \mathcal{P}^\top (\hat{A} + \hat{L})\eta(t) \\
%&\ \ \ \ \ + \phi^\top(t) \mathcal{B}^\top\mathcal{P}\eta(t) + \eta^\top(t) \mathcal{P}^\top\mathcal{B}\phi(t) \\
&= \eta^\top(t) \hat{A}^\top \mathcal{P} \eta(t) + \eta^\top(t) \hat{L}^\top \mathcal{P} \eta(t) + \eta^\top(t) \mathcal{P}^\top \hat{A} \eta(t) + \eta^\top(t) \mathcal{P}^\top \hat{L} \eta(t) \\ \notag
&\ \ \ \ \ + \phi^\top(t) \mathcal{B}^\top\mathcal{P}\eta(t) + \eta^\top(t) \mathcal{P}^\top\mathcal{B}\phi(t) \\ \notag
%&= \eta^\top(t) \left( \hat{A}^\top \mathcal{P} + \mathcal{P}^\top \hat{A} +  \mathcal{K} H + H^\top \mathcal{K}^\top  \right) \eta(t) \\ \notag
%&\ \ \ \ \ +  \phi^\top(t) \mathcal{B}^\top\mathcal{P}\eta(t) + \eta^\top(t) \mathcal{P}^\top \mathcal{B}\phi(t)\\ \notag
& \overset{\eqref{Lipschitz-geq-0},\eqref{Monotonicity-geq-0}}{\leq} \eta^\top(t) \left( \hat{A}^\top \mathcal{P} + \mathcal{P}^\top \hat{A} +  \mathcal{K} H + H^\top \mathcal{K}^\top \right) \eta(t)\\ \notag
&\ \ \ \ \ +  \phi^\top(t) \mathcal{B}^\top\mathcal{P}\eta(t) + \eta^\top(t) \mathcal{P}^\top\mathcal{B}\phi(t)  + \delta (\eta^\top(t) \mathcal{F}^\top\mathcal{F} \eta(t) - \phi_L^\top(t) \phi_L(t)) \\ \notag
&\ \ \ \ \ + ([J,0]\eta(t))^\top \Theta \phi_M(t) + \phi_M^\top (t) \Theta^\top [J,0]\eta(t) - \mu ([J,0]\eta(t))^\top ([J,0]\eta(t)) \\ \notag
&= \eta^\top(t) \left( \hat{A}^\top \mathcal{P} + \mathcal{P}^\top \hat{A} +  \mathcal{K} H + H^\top \mathcal{K}^\top + \delta \mathcal{F}^\top\mathcal{F}-\mu \cJ \right) \eta(t)\\ \notag
&\ \ \ \ \ +  \phi^\top(t) \mathcal{B}^\top\mathcal{P}\eta(t) + \eta^\top(t) \mathcal{P}^\top\mathcal{B}\phi(t) \\ \notag
&\ \ \ \ \ + \eta^\top(t) \hat{\Theta} \phi(t) + \phi^\top (t) \hat{\Theta}^\top \eta(t) - \delta \phi^\top(t) \Lambda_{q_L} \phi(t)\\
&= \begin{pmatrix} \eta(t) \\ \phi(t)\end{pmatrix}^\top \underset{=\cQ}{\underbrace{\begin{bmatrix} \hat{A}^\top\mathcal{P} + \mathcal{P}^\top\hat{A} + H^\top \mathcal{K}^\top + \mathcal{K}H+ \delta \mathcal{F}^\top\mathcal{F}-\mu \cJ & \mathcal{P}^\top\mathcal{B}+\hat{\Theta} \\ \mathcal{B}^\top\mathcal{P}+\hat{\Theta}^\top & -\delta \Lambda_{q_L} \end{bmatrix}}} \begin{pmatrix} \eta(t) \\ \phi(t)\end{pmatrix}.\label{time-derivative-V}
\end{align}

Let $S \in \mathbb{R}^{(n+k+q_L+q_M) \times n_{\mathcal{V}}}$ with orthonormal columns be such that  $\im S=\mathcal{V}^*_{\left[[\mathcal{E}, 0], [\mathcal{A}, \mathcal{B}] \right]}$ and $\rk(S)=n_{\mathcal{V}}$. Then inequality~\eqref{eq:Q-MI} reads $\hat{\mathcal{Q}}:=S^\top \mathcal{Q} S < 0$. Denote with $\lambda_{\mathcal{\hat{Q}}}^-$ the smallest eigenvalue of $-\mathcal{\hat{Q}}$, then $\lambda_{\mathcal{\hat{Q}}}^- >0$. Since $S$ has orthonormal columns we have $\|Sv\|~=~\|v\|$ for all $v~\in~\mathbb{R}^{n_{\mathcal{V}}}$.

By Lemma~\ref{Lem-augmented-trajectories-in-V} we have $\begin{smallpmatrix} \eta(t) \\ \phi(t) \end{smallpmatrix} \in \mathcal{V}^*_{\left[[\mathcal{E}, 0], [\mathcal{A}, \mathcal{B}] \right]}$ for all $t \geq t_0$, hence $\begin{smallpmatrix} \eta(t) \\ \phi(t) \end{smallpmatrix} = S v(t)$ for some $v\colon [t_0, \infty) \to \mathbb{R}^{n_{\mathcal{V}}}$.  Then~\eqref{time-derivative-V} becomes
\begin{equation}
\label{est-lyap}
\begin{aligned}
\forall\, t \geq t_0:\,\ddt \tilde{V}(\eta(t)) &\leq \begin{pmatrix} \eta(t) \\ \phi(t) \end{pmatrix}^\top  \mathcal{Q} \begin{pmatrix} \eta(t) \\ \phi(t) \end{pmatrix} = v^\top(t)  \mathcal{\hat{Q}} v(t) \\
& \leq - \lambda_{\mathcal{\hat{Q}}}^- \|v(t)\|^2 = - \lambda_{\mathcal{\hat{Q}}}^-
\left\|\begin{pmatrix} \eta(t) \\ \phi(t)\end{pmatrix}\right\|^2.
\end{aligned}
\end{equation}

Let $\overline{S} \in \mathbb{R}^{(n+k)\times n_{\overline{\mathcal{V}}}}$ with orthonormal columns be such that $\im \overline{S} =\overline{\mathcal{V}}^*_{\left[[\mathcal{E}, 0], [\mathcal{A}, \mathcal{B}] \right]}$ and $\rk(\overline{S})= n_{\overline{\mathcal{V}}}$. Then condition \eqref{eq:EP-MI} is equivalent to $\overline{S}^\top \mathcal{E}^\top \mathcal{P} \overline{S}>0$.
Since $\begin{smallpmatrix} \eta(t) \\ \phi(t) \end{smallpmatrix} \in \mathcal{V}^*_{\left[[\mathcal{E}, 0], [\mathcal{A}, \mathcal{B}] \right]}$ for all $t \geq t_0$ it is clear that $\eta(t) \in \overline{\mathcal{V}}^*_{\left[[\mathcal{E}, 0], [\mathcal{A}, \mathcal{B}] \right]}$ for all $t \geq t_0$. If $\overline{\mathcal{V}}^*_{\left[[\mathcal{E}, 0], [\mathcal{A}, \mathcal{B}] \right]} = \{0\}$ (which also holds when ${\mathcal{V}}^*_{\left[[\mathcal{E}, 0], [\mathcal{A}, \mathcal{B}] \right]}=\{0\}$), then this implies $\eta(t) = 0$, thus $e(t) = 0$ for all $t\ge t_0$, which completes the proof. Otherwise, $n_{\overline{\mathcal{V}}} > 0$ and we set $\eta(t) = \overline{S} \bar{\eta}(t)$ for some $\bar{\eta}\colon [t_0, \infty)\to \mathbb{R}^{n_{\overline{\mathcal{V}}}}$ and denote with $\lambda^+, \lambda^-$ the largest and smallest eigenvalue of $\overline{S}^\top \mathcal{E}^\top \mathcal{P} \overline{S}$, resp., where $\lambda^->0$ is a consequence of~\eqref{eq:EP-MI}. Then we have
\begin{equation}
\label{estimation-lyapunov}
\tilde{V}(\eta(t)) = \eta^\top(t) \mathcal{E}^\top\mathcal{P} \eta(t) = \bar{\eta}^\top(t) \overline{S}^\top\mathcal{E}^\top\mathcal{P}\overline{S}\bar{\eta}(t) \leq \lambda^+ \|\bar{\eta}(t)\|^2 = \lambda^+ \|\eta(t)\|^2
\end{equation}
and, analogously,
\begin{equation}\label{eq:est-lyap-2}
\forall\, t \geq t_0:\, \lambda^- \|\eta(t)\|^2 \leq \bar{\eta}^\top(t)\overline{S}^\top \mathcal{E}^\top\mathcal{P} \overline{S}\bar{\eta}(t) = \tilde{V}(\eta(t)) \leq \lambda^+ \|\eta(t)\|^2.
\end{equation}
Therefore,
\[
\forall\, t \geq t_0 : \, \ddt \tilde{V}(\eta(t)) \overset{(\ref{est-lyap})}{\leq} -\lambda_{\mathcal{\hat{Q}}}^- \left\|\begin{pmatrix} \eta(t) \\ \phi(t)\end{pmatrix}\right\|^2 \leq -\lambda_{\mathcal{\hat{Q}}}^- \|\eta(t)\|^2
\overset{(\ref{estimation-lyapunov})}{\leq} -\frac{\lambda_{\mathcal{\hat{Q}}}^-}{\lambda^+} \,\, \tilde{V}(\eta(t)).
\]
Now, abbreviate $\beta := \frac{\lambda_{\mathcal{\hat{Q}}}^-}{\lambda^+}$ and use Gronwall's Lemma to infer
\begin{equation}
\label{est-V}
\forall\, t \geq t_0:\, \tilde{V}(\eta(t)) \leq \tilde{V}(\eta(0)) e^{-\beta t}.
\end{equation}
Then we obtain
\[
\forall\, t\ge t_0:\ \|\eta(t)\|^2 \stackrel{\eqref{eq:est-lyap-2}}{\leq} \frac{1}{\lambda^-} \tilde{V}(\eta(t)) \overset{\eqref{est-V}}{\leq} \frac{\tilde{V}(\eta(0))}{\lambda^-} e^{-\beta t},
\]
and hence $\lim_{t\to\infty} e(t) = 0$, which completes the proof.
\end{proof}

\begin{Rem}\label{Rem-Thm1}\
\begin{enumerate}
\item Note that $\cA = \hat A + \hat{L}$, where $\hat{L} = [0_{(l+p)\times n}, \ast]$ is a solution of $\cP^\top \hat{L} = \cK H$ and hence the space ${\mathcal{V}_{\left[[\mathcal{E},0],[\mathcal{A},\mathcal{B}] \right]}^*}$ on which~\eqref{eq:Q-MI} is considered depends on the sought solutions~$\cP$ and~$\cK$ as well; using $\cP^\top\cA = \cP^\top \hat A + \cK H$, this dependence is still linear. Furthermore, note that~$\mathcal{K}$ only appears in union with the matrix $H = \left[\begin{smallmatrix} 0 &0\\0&I_k \end{smallmatrix}\right]$, thus only the last~$k$ columns of~$\mathcal{K}$ are of interest. In order to reduce the computational effort it is reasonable to fix the other entries beforehand, e.g.\ by setting them to zero.

\item We stress that the parameters in the description~\eqref{DAE} of the system are not entirely fixed, especially regarding the linear parts. More precisely, an equation of the form $\ddt E x(t) = Ax(t) + f(x(t),u(t))$, where~$f$ satisfies~\eqref{Lipschitz} can equivalently be written as $\ddt E x(t) = f_L(x(t),u(t))$, where $f_L(x,u) = Ax + f(x,u)$ also satisfies~\eqref{Lipschitz}, but with a different matrix~$F$. However, this alternative (with $A=0$) may not satisfy the necessary condition provided in Lemma~\ref{Lem:rank-E'-A'}, which hence should be checked in advance. Therefore, the system class~\eqref{DAE} allows for a certain flexibility and different choices of the parameters may or may not satisfy the assumptions of Theorem~\ref{Version-1}.

\item In the special case $E=0$, i.e., purely algebraic systems of the form $0=Ax(t)+Bf(x(t),u(t),y(t))$, Theorem~\ref{Version-1} may still be applicable. More precisely, condition~\eqref{eq:EP-MI} is satisfied in this case if, and only if, $\overline{\mathcal{V}}^*_{\left[[\mathcal{E}, 0], [\mathcal{A}, \mathcal{B}] \right]} = \{0\}$. This can be true, if for instance $\cB=0$ and $\cA$ has full column rank, because then $\overline{\mathcal{V}}^*_{\left[[\mathcal{E}, 0], [\mathcal{A}, \mathcal{B}] \right]} = [I_{n+k}, 0] \ker [\cA, 0] = \{0\}$.
\end{enumerate}
\end{Rem}

In the following theorem condition~\eqref{eq:EP-MI} is weakened to positive semi-definiteness. As a consequence, the system's matrices have to satisfy additional conditions, which are not present in Theorem~\ref{Version-1}. In particular, we require that~$\cE$ and~$\cA$ are square, which means that $k~=~l+p-n$. Furthermore, we require that $J G_M$ is invertible for a certain matrix~$G_M$ and that the norms corresponding to~$F$ and~$J$ are compatible if both kinds of nonlinearities are present.

\begin{Thm}\label{Version-2}
Use the notation from Theorem~\ref{Version-1} and set $k=l+p-n$. In addition, denote with $\mathcal{V}^*_{[\mathcal{E},\mathcal{A}]}, \mathcal{W}^*_{[\mathcal{E},\mathcal{A}]} \subseteq \mathbb{R}^{n+k}$ the Wong limits of the pencil $s\mathcal{E}-\mathcal{A} \in \mathbb{R}[s]^{(l+p)\times(n+k)}$ and let $V\in \mathbb{R}^{(n+k) \times n_{\mathcal{V}}}$ and $W \in \mathbb{R}^{(n+k) \times n_{\mathcal{W}}} $ be basis matrices of $\mathcal{V}^*_{[\mathcal{E},\mathcal{A}]}$ and $\mathcal{W}^*_{[\mathcal{E},\mathcal{A}]}$, resp., where $n_{\mathcal{V}}=\dim(\mathcal{V}^*_{[\mathcal{E},\mathcal{A}]})$ and $n_{\mathcal{W}}=\dim(\mathcal{W}^*_{[\mathcal{E},\mathcal{A}]})$. Furthermore, denote with $\lambda_{max}(M)$ the largest eigenvalue of a matrix $M$. \\

If there exist $\delta > 0$, $ \mathcal{P} \in \mathbb{R}^{(l+p) \times (n+k)}$ invertible and $\mathcal{K} \in \mathbb{R}^{(n+k) \times (n+k)}$ such that~\eqref{eq:Q-MI} holds and
\begin{equation}
\label{Conditions-Version-2}
	\begin{aligned}
&a)\ \mathcal{E}^\top\mathcal{P}=\mathcal{P}^\top\mathcal{E} \geq_{\overline{\mathcal{V}}^*_{\left[[\mathcal{E}, 0], [\mathcal{A}, \mathcal{B}] \right]}} 0,  \\
&b)\ \text{the pencil} \ s\mathcal{E} - \mathcal{A} \in \mathbb{R}[s]^{(l+p)\times(n+k)} \ \textit{is regular and its index is at most one},  \\
&c)\ F \ \text{is such that} \ \ \|F G_L\| < 1,\ \text{ where } \ G_L := -[I_{n},0]W[0,I_{n+k-r}][\mathcal{E}V,\mathcal{A}W]^{-1} \left[\begin{smallmatrix} B_L \\ 0 \end{smallmatrix}\right],\\
&d)\ JG_M \text{ is invertible and}\ \mu > \lambda_{\rm max}(\Gamma), \ \text{where}\ \ \Gamma := \tilde\Theta + \tilde\Theta^\top,  \  \tilde\Theta := \Theta (J G_M)^{-1},\\
&\quad G_M:= -[I_{n},0]W[0,I_{n+k-r}][\mathcal{E}V,\mathcal{A}W]^{-1} \begin{smallbmatrix} B_M \\ 0 \end{smallbmatrix},\\
&e)\ \text{there exists $\alpha>0$ such that $\|Fx\|\le \alpha \|Jx\|$ for all $x\in\R^n$ and, for}\\
&\quad\ \text{$S:= \tilde\Theta^\top (\Gamma-\mu I_{q_M})^{-1}\tilde\Theta$ we have}\\
&\quad \frac{\alpha \|JG_L\|}{1- \|F G_L\|} \left(\sqrt{\frac{\max\{0,\lambda_{\rm max}(S)\}}{\mu - \lambda_{\rm max}(\Gamma)}} + \|(\Gamma-\mu I_{q_M})^{-1}(\tilde\Theta^\top - \mu I_{q_M})\|\right) < 1,
	\end{aligned}
\end{equation}
then with $L_1 \in \mathbb{R}^{l\times k}$, $L_2 \in \mathbb{R}^{p \times k}$ such that $ \begin{smallbmatrix} 0 & L_1 \\ 0 & L_2 \end{smallbmatrix} = \mathcal{P}^{-\top} \mathcal{K} H$ the system \eqref{eq:state-estimator} is a state estimator\index{State estimator} for~\eqref{DAE}.
\end{Thm}

\begin{proof}
Assume \eqref{eq:Q-MI} and \eqref{Conditions-Version-2} a)-e) hold. Up to equation \eqref{est-V} the proof remains the same as for Theorem~\ref{Version-1}. By~\eqref{Conditions-Version-2} b) we may infer from~\cite[Thm. 2.6]{BergIlch12a} that there exist invertible $\mathcal{M} = \left[M_1^\top , M_2^\top\right]^\top \in \mathbb{R}^{(n+k)\times(l+p)}$ with $M_1 \in \mathbb{R}^{r \times (l+p)}$, $M_2 \in \mathbb{R}^{(n+k-r) \times (l+p)}$ and invertible $\mathcal{N} = \left[N_1 , N_2 \right] \in \mathbb{R}^{(n+k)\times (l+p)}$ with $N_1 \in \mathbb{R}^{(n+k) \times r}$, $N_2 \in \mathbb{R}^{(n+k) \times (l+p-r)}$ such that
\begin{equation}
\label{MEN}
\mathcal{M(E-A)N} = \begin{bmatrix} I_r - A_r & 0 \\ 0 & -I_{n+k-r} \end{bmatrix},
\end{equation}
where $r=\rk(\mathcal{E})$ and $A_r \in \mathbb{R}^{r \times r}$, and that
\begin{equation}
\label{chooseN}
\mathcal{N} = [V, W],\quad \mathcal{M} = [\mathcal{E}V, \mathcal{A}W]^{-1}.
\end{equation}
Let
\begin{equation}
\label{partP}
\mathcal{P} = \mathcal{M}^\top \begin{bmatrix} P_1 & P_2 \\ P_3 & P_4 \end{bmatrix} \mathcal{N}^{-1}
\end{equation}
with $P_1 \in \mathbb{R}^{n_{\mathcal{V}} \times n_{\mathcal{V}}}$, $P_4 \in \mathbb{R}^{n_{\mathcal{W}}\times n_{\mathcal{W}}}$ and $P_2, P_3^\top \in \mathbb{R}^{n_{\mathcal{V}} \times n_{\mathcal{W}}}$. Then condition \eqref{Conditions-Version-2} a) implies $P_1 > 0$ as follows. First, calculate
\begin{equation}
\label{EP=PE}
	\begin{aligned}
\mathcal{E}^\top\mathcal{P}
&= \mathcal{N}^{-T} \begin{bmatrix} I_r & 0 \\ 0 & 0 \end{bmatrix} \mathcal{M}^{-T} \mathcal{M}^{T} \begin{bmatrix} P_1 & P_2 \\ P_3 & P_4 \end{bmatrix} \mathcal{N}^{-1} = \mathcal{N}^{-T} \begin{bmatrix} P_1 & P_2 \\ 0 & 0 \end{bmatrix} \mathcal{N}^{-1} \\
	\end{aligned}
\end{equation}
which gives $P_2 = 0$ as $\mathcal{P}^\top\mathcal{E}=\mathcal{E}^\top\mathcal{P}$. Note that therefore~$P_1$ and~$P_4$ in~\eqref{partP} are invertible since~$\mathcal{P}$ is invertible by assumption. By \eqref{EP=PE} we have
\begin{equation}
\label{EP-VW}
	\begin{aligned}
\mathcal{E}^\top\mathcal{P}
& = \mathcal{N}^{-T} \begin{bmatrix} P_1 & 0 \\ 0 & 0 \end{bmatrix} \mathcal{N}^{-1} = [V, W]^{-T} \begin{bmatrix} P_1 & 0 \\ 0 & 0 \end{bmatrix} [V, W]^{-1}.
	\end{aligned}
\end{equation}
It remains to show $P_1 \geq 0$. Next, we prove the inclusion
\begin{equation}
\label{inclusion1}
\mathcal{V}^*_{[\mathcal{E},\mathcal{A}]}
\subseteq \overline{\mathcal{V}}^*_{\left[[\mathcal{E},0],[\mathcal{A},\mathcal{B}]\right]} = [I_{n+k},0]\mathcal{V}^*_{\left[[\mathcal{E},0],[\mathcal{A},\mathcal{B}]\right]}.
\end{equation}
To this end, we show $\mathcal{V}^{i}_{[\mathcal{E},\mathcal{A}]} \subseteq [I_{n+k, 0}]\mathcal{V}^i_{\left[[\mathcal{E},0],[\mathcal{A},\mathcal{B}]\right]}$ for all $i\in \mathbb{N}_0$. For $i=0$ this is clear. Now assume it is true for some $i \in \mathbb{N}_0$. Then
\begin{equation*}
	\begin{aligned}
[I_{n+k}, 0]\mathcal{V}^{i+1}_{\left[ [\mathcal{E},0],[\mathcal{A},\mathcal{B}] \right]} &= [I_{n+k},0][\mathcal{A},\mathcal{B}]^{-1}([\mathcal{E},0]\mathcal{V}^{i}_{\left[ [\mathcal{E},0],[\mathcal{A},\mathcal{B}] \right]}) \\
&= [I_{n+k}, 0]\setdef{\begin{pmatrix} \eta(t) \\ \phi(t)\end{pmatrix} \in \mathbb{R}^{n+k+q} }{ \mathcal{A}\eta + \mathcal{B} \phi \in \mathcal{E}\left([I_{n+k},0] \mathcal{V}^{i}_{\left[ [\mathcal{E},0],[\mathcal{A},\mathcal{B}] \right]}\right) }\\
&= \setdef{ \eta \in \mathbb{R}^{n+k} }{ \exists\, \phi \in \mathbb{R}^q:\ \mathcal{A}\eta + \mathcal{B}\phi \in \mathcal{E} \overline{\mathcal{V}}^{i}_{\left[ [\mathcal{E},0],[\mathcal{A},\mathcal{B}] \right]}  } \\
& \overset{\phi=0}{\supseteq} \setdef{\eta \in \mathbb{R}^{n+k} }{ \mathcal{A}\eta \in \mathcal{E} \overline{\mathcal{V}}^{i}_{\left[ [\mathcal{E},0],[\mathcal{A},\mathcal{B}] \right]} } \\
&= \mathcal{A}^{-1}\left( \mathcal{E}\overline{\mathcal{V}}^{i}_{\left[ [\mathcal{E},0],[\mathcal{A},\mathcal{B}] \right]}\right) \\
& \supseteq \mathcal{A}^{-1}\left( \mathcal{E} \mathcal{V}^i_{[\mathcal{E},\mathcal{A}]}  \right) = \mathcal{V}^{i+1}_{[\mathcal{E},\mathcal{A}]},
	\end{aligned}
\end{equation*}
which is the statement. Therefore it is clear that $\im V \subseteq \overline{\mathcal{V}}^*_{\left[[\mathcal{E},0],[\mathcal{A},\mathcal{B}]\right]} =\im \overline{V}$, with $\overline{V} \in \mathbb{R}^{(n+k) \times n_{\overline{\mathcal{V}}}}$ a basis matrix of $\overline{\mathcal{V}}^*_{\left[[\mathcal{E},0],[\mathcal{A},\mathcal{B}]\right]}$ and $n_{\overline{\mathcal{V}}} = \dim(\overline{\mathcal{V}}^*_{\left[[\mathcal{E},0],[\mathcal{A},\mathcal{B}]\right]})$. Thus there exists $R \in \mathbb{R}^{n_{\overline{\mathcal{V}}} \times n_{\mathcal{V}}}$ such that $V=\overline{V}R$. Now the inequality $\overline{V}^\top \mathcal{P}^\top\mathcal{E}\overline{V} \geq 0$ holds by condition~\eqref{Conditions-Version-2}~a) and implies
\begin{equation*}
	\begin{aligned}
0 \leq R^\top \overline{V}^\top \mathcal{P}^\top \mathcal{E} \overline{V} R = V^\top \mathcal{P}^\top \mathcal{E} V
&\,\,\,= \left([V,W] \begin{bmatrix} I_{n_{\mathcal{V}}} \\ 0 \end{bmatrix}\right)^\top \mathcal{P}^\top\mathcal{E}
\left([V,W] \begin{bmatrix} I_{n_{\mathcal{V}}} \\ 0 \end{bmatrix} \right) \\
&\overset{(\ref{EP-VW})}{=} [I_{n_{\mathcal{V}}}, 0] \begin{bmatrix} P_1 & 0 \\0 & 0 \end{bmatrix} \begin{bmatrix} I_{n_{\mathcal{V}}} \\ 0 \end{bmatrix} = P_1.
	\end{aligned}
\end{equation*}

Now, let $\mathcal{N}^{-1}\eta(t) = \begin{smallpmatrix} \eta_1(t) \\ \eta_2(t) \end{smallpmatrix}$, with $\eta_1(t) \in \mathbb{R}^{r}$ and $\eta_2(t) \in \mathbb{R}^{n+k-r}$ and consider the Lyapunov function $\tilde{V}(\eta(t))=\eta^\top(t) \mathcal{E}^\top\mathcal{P} \eta(t)$ in new coordinates:

\begin{equation}
\label{LYAPnew}
\begin{aligned}
\forall\, t \geq t_0:\, \tilde{V}(\eta(t))=\eta^\top(t) \mathcal{E}^\top\mathcal{P} \eta(t) &\overset{(\ref{EP=PE})}{=} \begin{pmatrix} \eta_1(t) \\ \eta_2(t) \end{pmatrix}^\top
\begin{bmatrix}
P_1 & 0 \\ 0 & 0
\end{bmatrix}
\begin{pmatrix} \eta_1(t) \\ \eta_2(t)\end{pmatrix} \\
&= \eta_1^\top(t) P_1 \eta_1(t) \geq \lambda_{P_1}^- \|\eta_1(t)\|^2,
\end{aligned}
\end{equation}
where $\lambda_{P_1}^- > 0 $ denotes the smallest eigenvalue of $P_1$. Thus \eqref{LYAPnew} implies
\begin{equation}
\label{eta1}
	\begin{aligned}
\forall\, t \geq t_0:\, \|\eta_1(t)\|^2 &\leq \frac{1}{\lambda_{P_1}^-} \eta^\top(t) \mathcal{E}^\top\mathcal{P} \eta(t) = \frac{1}{\lambda_{P_1}^-} \tilde{V}(\eta(t)) \overset{(\ref{est-V})}{\leq} \frac{\tilde{V}(\eta(0))}{\lambda_{P_1}^-} e^{- \beta t} \underset{t\to\infty}{\longrightarrow} 0.
	\end{aligned}
\end{equation}
Note that, if $\mathcal{V}^*_{[\mathcal{E},\mathcal{A}]} = \{0\}$, then $r=0$ and $\cN^{-1} \eta(t) = \eta_2(t)$, thus the above estimate~\eqref{eta1} is superfluous (and, in fact, not feasible) in this case; it is straightforward to modify the remaining proof to this case. With the aid of transformation~(\ref{MEN}) we have:
\begin{equation}
\label{eta2}
	\begin{aligned}
	 \mathcal{M} \ddt \mathcal{E} \eta(t) &= \mathcal{M} \mathcal{A}\eta(t) + \mathcal{M} \mathcal{B} \phi(t) \\
\iff \mathcal{M} \mathcal{E} \mathcal{N} \ddt \begin{pmatrix} \eta_1(t) \\ \eta_2(t)\end{pmatrix} &= \mathcal{M} \mathcal{A} \mathcal{N} \binom{\eta_1(t)}{\eta_2(t)} + \mathcal{M} \mathcal{B} \phi(t) \\
\iff \begin{bmatrix} I_r & 0 \\ 0 & 0 \end{bmatrix} \ddt \begin{pmatrix} \eta_1(t) \\ \eta_2(t)\end{pmatrix} &= \begin{bmatrix} A_r & 0 \\ 0 & I_{n+k-r} \end{bmatrix} \begin{pmatrix} \eta_1(t) \\ \eta_2(t)\end{pmatrix} + \begin{bmatrix} M_1 \\ M_2 \end{bmatrix} \mathcal{B} \phi(t),
	\end{aligned}
\end{equation}
from which it is clear that $\eta_2(t) = - M_2 \mathcal{B} \phi(t)$. Observe
\[
e(t)=[I_n,0]\eta(t) = [I_n,0]\mathcal{N} \begin{pmatrix} \eta_1(t) \\ \eta_2(t)\end{pmatrix} = [I_n,0]V \eta_1(t) + [I_n,0]W \eta_2(t) =: e_1(t) + e_2(t),
\]
where $\lim_{t\to\infty} e_1(t) = 0$ by~\eqref{eta1}. We show $e_2(t)~=~-[I_n,0] W M_2 \mathcal{B} \phi(t)~\to~0 $ for $t~\to~\infty$. Set
\[
    e_2^L(t):= G_L\phi_L(t),\quad e_2^M(t) := G_M\phi_M(t)
\]
so that $e_2(t) = e_2^L(t) + e_2^M(t)$. Next, we inspect the Lipschitz condition~\eqref{Lipschitz}:
\begin{equation*}
	\begin{aligned}
\|\phi_L(t)\| &\leq \|F e(t)\| \le \|F e_1(t)\| + \| F e_2^L(t)\| + \| F e_2^M(t)\| \\
&\le \|F e_1(t)\| + \|F G_L\| \|\phi_L(t)\| + \| F e_2^M(t)\|,
	\end{aligned}
\end{equation*}
which gives, invoking~\eqref{Conditions-Version-2} c),
\begin{equation}\label{eq:Lip-phiL}
    \|\phi_L(t)\|  \leq  \big( 1- \|F G_L\| \big)^{-1} \big( \|F e_1(t)\| + \| F e_2^M(t)\| \big).
\end{equation}
Set $\hat e(t) := e_1(t) + e_2^L(t) = e_1(t) + G_L\phi_L(t)$  and $\kappa := \frac{\alpha \|JG_L\|}{1- \|F G_L\|}$ and observe that~\eqref{eq:Lip-phiL} together with~\eqref{Conditions-Version-2} e) implies
\begin{equation}\label{eq:est-hate}
    \|J \hat e(t)\|\le \|J e_1(t)\| + \|JG_L\| \|\phi_L(t)\| \stackrel{\eqref{eq:Lip-phiL}}{\le} (1+\kappa) \|J e_1(t)\| + \kappa \|J e_2^M(t)\|.
\end{equation}
Since $J G_M$ is invertible by~\eqref{Conditions-Version-2} d) we find that
\begin{equation}\label{eq:phiM-e2M}
     \phi_M(t) = (J G_M)^{-1} J e_2^M(t),
\end{equation}
and hence the monotonicity condition~\eqref{Monotonicity} yields, for all $t \geq t_0$,
\begin{equation*}
\label{LHS-Monotonicity}
\begin{aligned}
\mu \norm{J{e}(t)}^2 &\le (J e(t))^\top \Theta \phi_M(t) + \phi^\top_M(t) \Theta^\top J {e}(t) \\
&= (J\hat e(t)  + Je_2^M(t))^\top \tilde{\Theta} Je_2^M(t) + ( Je_2^M(t))^\top \tilde{\Theta}^\top (J\hat e(t) + Je_2^M(t)) \\
&= 2 (J\hat e(t))^\top \tilde{\Theta}J e_2^M(t) + (Je_2^M(t))^\top \big(\underbrace{\tilde{\Theta} + \tilde{\Theta}^\top}_{=\Gamma}\big) Je_2^M(t)
\end{aligned}
\end{equation*}
and on the left-hand side
\begin{equation*}
\label{RHS-Monotonicity}
\begin{aligned}
\mu \norm{J{e}(t)}^2 = \mu \left( \norm{J\hat e(t)}^2 + \norm{Je_2^M(t)}^2 + 2(J\hat e(t))^\top (Je_2^M(t)) \right).
\end{aligned}
\end{equation*}
Therefore, we find that
\begin{align*}
  0 &\leq
-\mu \norm{Je_2^M(t)}^2 -\mu \norm{J\hat e(t)}^2
- 2 \mu (J\hat e(t))^\top (Je_2^M(t))\\
&\quad + 2 (J\hat e(t))^\top \tilde{\Theta} (J e_2^M(t)) + (Je_2^M(t))^\top \Gamma (Je_2^M(t)) \\
&= \begin{pmatrix} J \hat e(t)\\ Je_2^M(t)\end{pmatrix}^\top \begin{bmatrix} -\mu I_{q_M} & \tilde\Theta - \mu I_{q_M}\\ \tilde\Theta^\top -\mu I_{q_M} & \Gamma - \mu I_{q_M}\end{bmatrix} \begin{pmatrix} J\hat e(t)\\ Je_2^M(t)\end{pmatrix}.
\end{align*}
Since $\Gamma-\mu I_{q_M}$ is invertible by~\eqref{Conditions-Version-2} d) we may set $\Xi:= \tilde\Theta^\top -\mu I_{q_M}$ and $\tilde e_{2}^M(t) := (\Gamma-\mu I_{q_M})^{-1} \Xi J \hat e(t) + J e_2^M(t)$. Then
\begin{align*}
    0 &\le \begin{pmatrix} J\hat e(t)\\ J e_2^M(t)\end{pmatrix}^\top \begin{bmatrix} -\mu I_{q_M} & \tilde\Theta - \mu I_{q_M}\\ \tilde\Theta^\top -\mu I_{q_M} & \Gamma - \mu I_{q_M}\end{bmatrix} \begin{pmatrix} J \hat e(t)\\ J e_2^M(t)\end{pmatrix} \\
    &= \begin{pmatrix} J \hat e(t)\\ \tilde e_{2}^M(t)\end{pmatrix}^\top \begin{bmatrix} -\mu I_{q_M} - \Xi^\top (\Gamma-\mu I_{q_M})^{-1} \Xi & 0\\ 0 & \Gamma - \mu I_{q_M}\end{bmatrix} \begin{pmatrix} J \hat e(t)\\ \tilde e_{2}^M(t)\end{pmatrix}.
\end{align*}
Therefore, using $\mu - \lambda_{\rm max}(\Gamma) >0$ by~\eqref{Conditions-Version-2} d) and computing
\[
    -\mu I_{q_M} - \Xi^\top (\Gamma-\mu I_{q_M})^{-1} \Xi = \tilde \Theta^\top (\tilde\Theta + \tilde\Theta^\top - \mu I_{q_M})^{-1} \tilde\Theta = S,
\]
we obtain
\[
    0\le \max\{0,\lambda_{\rm max}(S)\} \|J \hat e(t)\|^2 - (\mu - \lambda_{\rm max}(\Gamma)) \|\tilde e_{2}^M(t)\|^2,
\]
which gives
\begin{align*}
    \|Je_2^M(t)\| & \le \|(\Gamma-\mu I_{r_M})^{-1} \Xi\| \|J \hat e(t)\| + \|\tilde e_{2}^M(t)\|\\
    &\le  \left(\sqrt{\frac{\max\{0,\lambda_{\rm max}(S)\}}{\mu - \lambda_{\rm max}(\Gamma)}} + \|(\Gamma-\mu I_{r_M})^{-1} \Xi\|\right) \|J \hat e(t)\|\\
    &\stackrel{\eqref{eq:est-hate}}{\le} \left(\sqrt{\frac{\max\{0,\lambda_{\rm max}(S)\}}{\mu - \lambda_{\rm max}(\Gamma)}} + \|(\Gamma-\mu I_{r_M})^{-1} \Xi\|\right) \big( (1+\kappa) \|J e_1(t)\| + \kappa \| Je_2^M(t)\|\big).
\end{align*}
It then follows from~\eqref{Conditions-Version-2} e) that $\lim_{t\to\infty} J e_2^M(t) = 0$, and additionally invoking~\eqref{eq:Lip-phiL} and~\eqref{eq:phiM-e2M} gives $\lim_{t\to\infty} \phi_L(t) = 0$ and $\lim_{t\to\infty} \phi_M(t) = 0$, thus $\norm{e_2(t)} \leq \norm{G_L \phi_L(t)} + \norm{G_M \phi_M(t)} \underset{t\to\infty}{\longrightarrow}~0$, and finally $\lim_{t\to\infty} e(t) = 0$.
\end{proof}

\begin{Rem}\ \label{Rem-Thm2}
\begin{enumerate}
\item If the nonlinearity~$f$ in~\eqref{DAE} consists only of $f_L$ satisfying the Lipschitz condition, then the conditions~\eqref{Conditions-Version-2} d) and~e) are not present. If it consists only of the monotone part~$f_M$, then the conditions~\eqref{Conditions-Version-2} c) and~e) are not present. In fact, condition~\eqref{Conditions-Version-2}~e) is a ``mixed condition'' in a certain sense which states additional requirements on the combination of both (classes of) nonlinearities.

\item The following observation is of practical interest. Whenever~$f_L$ satisfies~\eqref{Lipschitz} with a certain matrix~$F$, it is obvious that~$f_L$ will satisfy~\eqref{Lipschitz} with any other~$\tilde{F}$ such that $\norm{F} \leq \norm{\tilde{F}}$. However, condition~\eqref{Conditions-Version-2} c) states an upper bound on the possible choices of~$F$. Similarly, if~$f_M$ satisfies~\eqref{Monotonicity} with certain~$\Theta$ and~$\mu$, then~$f_M$ satisfies~\eqref{Monotonicity} with any~$\tilde{\mu} \leq \mu$, for a fixed~$\Theta$. On the other hand, condition~\eqref{Conditions-Version-2} d) states lower bounds for~$\mu$ (involving~$\Theta$ as well). Additional bounds are provided by \eqref{eq:Q-MI} and condition~\eqref{Conditions-Version-2} e). Analogous thoughts hold for the other parameters. Hence~$F$, $\delta$, $J$, $\Theta$ and $\mu$ can be utilized in solving the conditions of Theorems~\ref{Version-1} and~\ref{Version-2}.
\item The condition $\|Fx\| \leq \alpha \|Jx\|$ from \eqref{Conditions-Version-2}~e) is quite restrictive since it connects the Lipschitz estimation of $f_L$ with the domain of $f_M$. This relation is far from natural and relaxing it is a topic of future research. The inequality would always be satisfied for $J=I_n$ by taking $\alpha = \|F\|$, however in view of the monotonicity condition~\eqref{Monotonicity}, the inequality~\eqref{eq:Q-MI} and conditions~\eqref{Conditions-Version-2} this would be even more restrictive.

\item In the case $E=0$ the assumptions of Theorem~\ref{Version-2} simplify a lot. In fact, we may calculate that in this case we have $\mathcal{V}_{\left[[\mathcal{E},0],[\mathcal{A},\mathcal{B}] \right]}^* = \ker [\cA, \cB]$ and hence the inequality~\eqref{eq:Q-MI} becomes
\begin{align*}
& \cQ = \begin{bmatrix} \mathcal{A}^\top \mathcal{P} + \mathcal{P}^\top \mathcal{A}+ \delta \mathcal{F}^\top \mathcal{F}-\mu \cJ & \mathcal{P}^\top \mathcal{B}+\hat{\Theta} \\ \mathcal{B}^\top\mathcal{P}+\hat{\Theta}^\top & - \delta \Lambda_{q_L} \end{bmatrix} <_{\ker [\cA, \cB]} 0 \\
  \iff \quad &\forall \begin{smallpmatrix} \eta\\ \phi \end{smallpmatrix} \in \ker [\cA, \cB]:\ \ \begin{smallpmatrix} \eta\\ \phi \end{smallpmatrix}^\top \cQ \begin{smallpmatrix} \eta\\ \phi \end{smallpmatrix} < 0  \\
\iff\quad  &\forall \begin{smallpmatrix} \eta\\ \phi \end{smallpmatrix} \in \ker [\cA, \cB]:\ \ \eta^\top\left(\mathcal{A}^\top \mathcal{P} + \mathcal{P}^\top\mathcal{A} + \delta \mathcal{F}^\top \mathcal{F} - \mu \cJ\right) \eta - \delta \norm{\Lambda_{q_L} \phi}^2 \\
 &+ \phi^\top \left(\mathcal{B}^\top \mathcal{P} + \hat{\Theta}\right) \eta + \eta^\top \left(\mathcal{P}^\top \mathcal{B} + \hat{\Theta}^\top \right) \phi < 0  \\
  \iff\quad  &\forall \begin{smallpmatrix} \eta\\ \phi \end{smallpmatrix} \in \ker [\cA, \cB]:\ \ (\underbrace{\eta^\top \mathcal{A}^\top + \phi^\top(t) \mathcal{B}^\top}_{=0}) \mathcal{P} \eta + \eta^\top \mathcal{P}^\top (\underbrace{\mathcal{A}\eta + \mathcal{B} \phi}_{=0})\\
 &+ \delta \big(\norm{\mathcal{F}\eta}^2 - \norm{\Lambda_{q_L} \phi}^2 \big) + \eta^\top \hat{\Theta} \phi + \phi^\top \hat{\Theta}^\top \eta - \mu \eta^\top \cJ \eta < 0.
\end{align*}
Now,  $\mathcal{A}$ is invertible by \eqref{Conditions-Version-2} b) and hence $\eta = -\mathcal{A}^{-1} \mathcal{B} \phi$. Therefore, the inequality~\eqref{eq:Q-MI} is equivalent to
\begin{multline*}
\delta\left( (\mathcal{F}\mathcal{A}^{-1} \mathcal{B})^\top (\mathcal{F}\mathcal{A}^{-1} \mathcal{B}) - \Lambda_{q_L}\right) - (\mathcal{A}^{-1} \mathcal{B})^\top \hat\Theta - \hat{\Theta}^\top \mathcal{A}^{-1} \mathcal{B} \\
- \mu  (\mathcal{A}^{-1} \mathcal{B})^\top \cJ (\mathcal{A}^{-1} \mathcal{B}) < 0,
\end{multline*}
which is of a much simper shape.

%\item Note, that for $k>l+p$ condition~\eqref{eq:EP-MI} cannot be satisfied. \tb{Das stimmt nicht, nimm $\cA = [\cA_1 // 0]$ und $\cB=[I // I]$, dann ist $\ker [\cA, \cB] = \{0\}$}
%
%    This can easily be seen as follows. Assume $k>l+p$. Then $n+k>l+p$ and since $\ker\cA \subseteq \cV^*_{[\cE,\cA]} \subseteq \overline{\cV}^*_{\left[[\cE,0],[\cA,\cB]\right]}$ it holds
%\[
%n_{\bar{\cV}} = \dim(\overline{\cV}^*_{\left[[\cE,0],[\cA,\cB]\right]}) \geq n+k-(l+p) > n.
%\]
%But on the other hand for $\im \overline{S} = \overline{\cV}^*_{\left[[\cE,0],[\cA,\cB]\right]}$ it holds $\rk(\overline{S}^\top \cE^\top \cP \overline{S})~\leq~\min\{n,l,n_{\bar{\cV}}\}$ and thence~\eqref{eq:EP-MI} cannot be satisfied.

\item The conditions presented in Theorems~\ref{Version-1} and \ref{Version-2} are sufficient conditions only. The following example does not satisfy the conditions in the theorems but a state estimator exists for it. Consider $\dot{x}(t) = -x(t)$, $y(t)=0$, then the system $\dot{z}(t) = -z(t)$ , $0=d_1(t) - d_2(t)$ of the form~\eqref{eq:state-estimator} with $L_1 = [0,0]$ and $L_2 = [1, -1]$ is obviously a state estimator, since the first equation is independent of the innovations~$d_1$,~$d_2$ and solutions satisfy $\lim_{t\to\infty} z(t) - x(t) = 0$. However, we have $n + k = 3 >2=l+p$ and therefore Theorem~\ref{Version-2} is not applicable. Furthermore, the assumptions of Theorem~\ref{Version-1} are not satisfied since 
    \[
        \overline{\cV}^*_{[[\cE,0],[\cA,\cB]]} = \cV^*_{[\cE,\cA]} = \im \begin{bmatrix} 1&0\\ 0&1\\ 0&1\end{bmatrix}\quad \text{and}\quad \cE \cV =  \begin{bmatrix} 1&0\\ 0&0\end{bmatrix},
    \]
    by which $\ker \cE \cV \neq \{0\}$ and hence~\eqref{eq:EP-MI} cannot be true. We also like to stress that therefore, in virtue of Lemma~\ref{Lem:rank-E'-A'}, $n\le l+p$ is a necessary condition for the existence of a state estimator of the form~\eqref{eq:state-estimator}, but $n+k \le l+p$ is not.
\end{enumerate}
\end{Rem}

%\begin{Rem}\
%The conditions presented in Theorems \ref{Version-1} and \ref{Version-2} are sufficient conditions only. The following short example gives a state estimator which does not satisfy \eqref{eq:EP-MI}. Consider a system $\dot{x} = -x$, $y=0$. Then the system $\dot{z} = -z$, $0=d_1 - d_2$ is obviously a state estimator, choosing $d_2 = d_1$. Furthermore, even $n+k=3>2=l+p$.
%\end{Rem}

%%%
\section{Sufficient conditions for asymptotic observers}\label{Sufficient conditions for asymptotic observers}\index{Asymptotic observer}
%%%

In the following theorem some additional conditions are asked to be satisfied in order to guarantee that the resulting observer candidate is in fact an asymptotic observer\index{Asymptotic observer}, i.e., it is a state estimator and additionally satisfies~\eqref{eq-observer}. To this end, we utilize an implicit function theorem from~\cite{GutuJara07}.\\

\begin{Thm}\label{Version-3}
Use the notation from Theorem~\ref{Version-2} and assume that $\cX=\R^n$, $\cU=\R^m$ and $\cY=\R^p$. Additionally, let $\mathcal{M}$, $\mathcal{N} \in \mathbb{R}^{(n+k)\times(l+p)}$ be as in~\eqref{chooseN}, set $\bar{\mathcal{N}} := [I_n,0]\mathcal{N}$ and $\left[ \begin{smallmatrix} \hat{B}_1 \\ \hat{B}_2 \end{smallmatrix}\right] := \mathcal{M} \left[\begin{smallmatrix} B_L&B_M&0\\0&0&I_p \end{smallmatrix}\right]$, where $\hat{B}_1 \in \mathbb{R}^{r \times (q_L+q_M+p)}$, $\hat{B}_2 \in \mathbb{R}^{(n+k-r) \times (q_L+q_M+p)}$. Let
\begin{equation*}
	\begin{aligned}
G\colon \mathbb{R}^r \times \mathbb{R}^{n+k-r} \times \mathbb{R}^m \times \mathbb{R}^p \to \mathbb{R}^{n+k-r},\
(x_1,x_2,u,y) \mapsto  x_2 + \hat{B}_2 \left( \begin{matrix}f_L\big(\bar{\mathcal{N}}\binom{x_1}{x_2},u,y\big)\\f_M\big(J\bar{\mathcal{N}}\binom{x_1}{x_2},u,y\big)\\h(u) - y \end{matrix} \right)
	\end{aligned}
\end{equation*}
and $Z_0 := \setdef{ (x_1,x_2,u,y) \in \mathbb{R}^r \times \mathbb{R}^{n+k-r} \times \mathbb{R}^m \times \mathbb{R}^p }{ G(x_1,x_2,u,y) = 0 }$.

If there exist $\delta > 0$, $ \mathcal{P} \in \mathbb{R}^{(l+p) \times (n+k)}$ invertible and $\mathcal{K} \in \mathbb{R}^{(l+p) \times (n+k)}$ such that~\eqref{eq:Q-MI} and~\eqref{Conditions-Version-2} hold and in addition
\begin{equation}
\label{condition-B-3}
	\begin{aligned}
&a)\ \tfrac{\partial}{\partial x_2} G (x_1,x_2,u,y) \text{ is invertible for all } (x_1,x_2,u,y) \in Z_0, \\
&b)\ \text{there exists}  \ \omega\in\cC([0, \infty) \to (0,\infty)) \   \text{nondecreasing with} \
\int_0^\infty \frac{{\rm d}t}{\omega(t)} = \infty \ \   \text{such that}\\
&\quad\ \ \ \forall\, (x_1,x_2,u,y) \in Z_0:\ \norm{\left( \tfrac{\partial}{\partial x_2} G (x_1,x_2,u,y) \right)^{-1}}\norm{\tfrac{\partial}{\partial (x_1,u,y)} G (x_1,x_2,u,y) } \leq \omega(\norm{x_2}),\\
&c)\ Z_0 \text{ is connected}, \\
&d)\ f_M \ \ \text{is locally Lipschitz continuous in the first variable},
	\end{aligned}
\end{equation}
then with $L_1 \in \mathbb{R}^{l\times k}$, $L_2 \in \mathbb{R}^{p \times k}$ such that $\left[ \begin{smallmatrix} 0 & L_1 \\ 0 & L_2 \end{smallmatrix} \right] = \mathcal{P}^{-\top} \mathcal{K} H$ the system \eqref{eq:state-estimator} is an asymptotic observer\index{Asymptotic observer} for~\eqref{DAE}.
\end{Thm}

\begin{proof}
Since~\eqref{eq:state-estimator} is a state estimator for~\eqref{DAE} by Theorem~\ref{Version-2}, it remains to show that~\eqref{eq-observer} is satisfied. To this end, let $I \subseteq \mathbb{R}$ be an open interval, $t_0 \in I$, and
$(x,u,y,z,d)\in \cC(I \to \mathbb{R}^n \times \mathbb{R}^m \times \mathbb{R}^p \times \mathbb{R}^n \times \mathbb{R}^k)$
such that $(x,u,y) \in \mathfrak{B}_{(\ref{DAE})}$ and $\left( \binom{z}{d}, \binom{u}{y}, z \right) \in \mathfrak{B}_{(\ref{eq:state-estimator})}$. Recall that $B = [B_L, B_M]$ and $f(x,u,y) = \begin{smallpmatrix} f_L(x,u,y)\\ f_M(Jx,u,y)\end{smallpmatrix}$. Now assume $Ex(t_0) = E z(t_0)$ and recall the equations\\
\begin{minipage}[c]{1.0\textwidth}
	\begin{minipage}[c]{0.43\textwidth}
\begin{equation*}
	\begin{aligned}
\ddt Ex(t) &= A x(t) + B f(x(t),u(t),y(t)), \\
y(t) &= Cx(t) + h(u(t)),
	\end{aligned}
\end{equation*}
	\end{minipage}%
	\hfill
	\begin{minipage}[c]{0.53\textwidth}
\begin{equation*}
	\begin{aligned}
\ddt Ez(t) &= A z(t) + B f(z(t),u(t),y(t)) + L_1 d(t), \\
0 &= Cz(t) -y(t) + h(u(t))+ L_2 d(t).
	\end{aligned}
\end{equation*}
	\end{minipage}%
\end{minipage}
\ \\
This is equivalent to
\begin{equation}
\label{recall-DAE}
	\begin{aligned}
\ddt \mathcal{E} \begin{pmatrix} x(t) \\ 0\end{pmatrix} &= \mathcal{A}\binom{x(t)}{0} + \begin{bmatrix} B & 0 \\ 0 & I_p \end{bmatrix} \begin{pmatrix} f(x(t),u(t),y(t)) \\ h(u(t)) - y(t)\end{pmatrix}
	\end{aligned}
\end{equation}
and
\begin{equation*}
	\begin{aligned}
\ddt \mathcal{E} \begin{pmatrix} z(t) \\ d(t) \end{pmatrix} &= \mathcal{A}\binom{z(t)}{d(t)} + \begin{bmatrix} B & 0 \\ 0 & I_p \end{bmatrix} \begin{pmatrix} f(x(t),u(t),y(t)) \\ h(u(t)) - y(t)\end{pmatrix}.
	\end{aligned}
\end{equation*}
Let $\begin{smallpmatrix} x_1(t)\\ x_2(t)\end{smallpmatrix} = \mathcal{N}^{-1} \begin{smallpmatrix} x(t)\\ 0\end{smallpmatrix}$ and $\begin{smallpmatrix} z_1(t)\\ z_2(t)\end{smallpmatrix} = \mathcal{N}^{-1} \begin{smallpmatrix} z(t)\\ d(t)\end{smallpmatrix}$. Application of transformations~\eqref{MEN} to~\eqref{recall-DAE} gives
\begin{equation*}
	\begin{aligned}
\begin{bmatrix} I_r & 0 \\ 0 & 0 \end{bmatrix} \begin{pmatrix} \dot{x_1}(t) \\ \dot{x_2}(t) \end{pmatrix}  = \begin{bmatrix} A_r & 0 \\ 0 & I_{n+k-r} \end{bmatrix} \begin{pmatrix} x_1(t) \\ x_2(t) \end{pmatrix} &+ \begin{bmatrix} \hat{B}_1 \\ \hat{B}_2 \end{bmatrix} \begin{pmatrix}f(\bar{\mathcal{N}}\binom{x_1(t)}{x_2(t)},u(t),y(t))\\h(u(t)) - y(t)\end{pmatrix} \\
	\end{aligned}
\end{equation*}
or, equivalently,
\begin{equation*}
	\begin{aligned}
\dot{x_1}(t) &= A_r x_1(t) + \hat{B}_1
\begin{pmatrix} f(\bar{\mathcal{N}}\binom{x_1(t)}{x_2(t)},u(t),y(t)) \\ h(u(t)) - y(t)\end{pmatrix} \\
0 &=\underbrace{ x_2(t) + \hat{B}_2
\begin{pmatrix} f(\bar{\mathcal{N}}\binom{x_1(t)}{x_2(t)},u(t),y(t)) \\ h(u(t)) - y(t) \end{pmatrix}}_{= G(x_1(t),x_2(t),u(t),y(t))}
	\end{aligned}
\end{equation*}
with $\bar{\mathcal{N}} := [I_n, 0] \mathcal{N}$ and $\mathcal{M} \left[ \begin{smallmatrix} B & 0 \\ 0 & I_p \end{smallmatrix} \right] = \left[ \begin{smallmatrix} \hat{B}_1 \\ \hat{B}_2 \end{smallmatrix} \right]$.\\
Since \eqref{condition-B-3}~a)--c) hold, the global implicit function theorem in~\cite[Cor.~5.3]{GutuJara07} ensures the existence of a unique continuous map $g\colon \mathbb{R}^r \times \mathbb{R}^m \times \mathbb{R}^p \to \mathbb{R}^{n+k-r}$ such that $G(x_1,g(x_1,u,y),u,y)=~0$ for all $(x_1,u,y) \in \mathbb{R}^r \times \mathbb{R}^m \times \mathbb{R}^p$, and hence $x_2(t)=g(x_1(t),u(t),y(t))$ for all $t \in I$.
Thus~$x_1$ solves the ordinary differential equation
\begin{equation}
\label{ODE}
\dot{x}_1(t) = A_r x_1(t) + \hat{B}_1 \left( \begin{matrix}f(\bar{\mathcal{N}}\binom{x_1(t)}{g(x_1(t),u(t),y(t))},u(t),y(t))\\h(u(t)) - y(t)\end{matrix}\right)
\end{equation}
with initial value $x_1(t_0)$ for all $t \in I$; and $z_1(t)$ solves the same ODE with same initial value $z_1(t_0)=x_1(t_0)$. This can be seen as follows: $Ex(t_0)=Ez(t_0)$ implies $\mathcal{E}\begin{smallpmatrix} x(t_0)\\ 0\end{smallpmatrix} = \mathcal{E}\begin{smallpmatrix} z(t_0)\\ d(t_0)\end{smallpmatrix}$, and the transformation (\ref{MEN}) gives
\begin{equation*}
	\begin{aligned}
	\mathcal{E} \begin{pmatrix} x(t_0) \\ 0 \end{pmatrix}	 &= \mathcal{M}^{-1} \begin{bmatrix} I_r &0\\0&0	\end{bmatrix} \mathcal{N}^{-1} \begin{pmatrix} x(t_0) \\ 0 \end{pmatrix} = \mathcal{M}^{-1} \begin{bmatrix} I_r &0\\0&0	\end{bmatrix} \begin{pmatrix} x_1(t_0) \\ x_2(t_0) \end{pmatrix} = \mathcal{M}^{-1} \begin{pmatrix} x_1(t_0) \\ 0 \end{pmatrix}, \\
	\mathcal{E} \begin{pmatrix} z(t_0) \\ d(t_0) \end{pmatrix} &= \mathcal{M}^{-1} \begin{bmatrix} I_r &0\\0&0	\end{bmatrix} \mathcal{N}^{-1} \begin{pmatrix} z(t_0) \\ d(t_0) \end{pmatrix} = \mathcal{M}^{-1} \begin{bmatrix} I_r &0\\0&0	\end{bmatrix} \begin{pmatrix} z_1(t_0) \\ z_2(t_0) \end{pmatrix} = \mathcal{M}^{-1}\begin{pmatrix} z_1(t_0) \\ 0 \end{pmatrix},
	\end{aligned}
\end{equation*}
which implies $x_1(t_0) = z_1(t_0)$.

Furthermore, $g(x_1,u,y)$ is differentiable, which follows from the properties of $G$: Let $v=(x_1,u,y)$ and write $G(x_1,g(v),u,y)=\tilde{G}(v,g(v))$, then taking the derivative yields
\begin{equation*}
	\begin{aligned}
	\frac{\rm{d}}{\rm{dv}}\tilde{G}(v,g(v)) &= \frac{\partial}{\partial v}\tilde G(v,g(v)) + \frac{\partial}{\partial g}\tilde G(v,g(v))g'(v) = 0 \\
	& \Rightarrow g'(v) = - \left( \frac{\partial \tilde G(v,g(v))}{\partial g} \right)^{-1} \left( \frac{\partial \tilde G(v,g(v))}{\partial v}\right),
	\end{aligned}
\end{equation*}
which is well defined by assumption. Hence $g(x_1,u,y)$ is in particular locally Lipschitz. Since~$f_L$ is globally Lipschitz in the first variable by~\eqref{Lipschitz} and~$f_M$ is locally Lipschitz in the first variable by assumption~\eqref{condition-B-3}~d), $(x_1,u,y) \mapsto f\left(\bar{\mathcal{N}}\binom{x_1(t)}{g(x_1(t),u(t),y(t))},u(t),y(t)\right)$ is locally Lipschitz in the first variable and therefore the solution of~\eqref{ODE} is unique by the Picard-Lindel\"of theorem, see e.g.~\cite[Thm.~4.17]{LogeRyan14}; this implies $z_1(t) = x_1(t)$ for all $t \in I$. Furthermore,
\begin{equation*}
x_2(t) = g(x_1(t),u(t),y(t)) = g(z_1(t),u(t),y(t)) = z_2(t)
\end{equation*}
for all $t \in I$, and hence~\eqref{eq:state-estimator} is an observer for~\eqref{DAE}. Combining this with the fact that~\eqref{eq:state-estimator} is already a state estimator,~\eqref{eq:state-estimator} is an asymptotic observer for~\eqref{DAE}.
\end{proof}

\section{Examples} \label{examples}

We present some instructive examples to illustrate Theorems~\ref{Version-1},~\ref{Version-2} and~\ref{Version-3}. Note that the inequality~\eqref{eq:Q-MI} does not have unique solutions $\mathcal{P}$ and $\mathcal{K}$ and hence the resulting state estimator is just one possible choice. The first example illustrates Theorem~\ref{Version-1}.

\begin{Ex}\label{Ex1}
Consider the DAE
\begin{equation}
\label{Ex-Thm4_1_n+k<l+p}
\begin{aligned}
	\ddt \begin{bmatrix} 1 & 0 \\ 1&1\\ 0&0 \\ 0&1 \end{bmatrix} \begin{pmatrix}
 x_1(t) \\ x_2(t) \end{pmatrix}
	&=\begin{bmatrix} 0& -3 \\-2&0\\\ 1 &-2 \\ 0&0 \end{bmatrix} \begin{pmatrix} x_1(t) \\ x_2(t) \end{pmatrix}
	+ \begin{bmatrix} 0&2\\1&-1\\0&1\\1&0\end{bmatrix}\begin{pmatrix}\sin(x_1(t) - x_2(t))
	\\  x_2(t) + \exp( x_2(t)) \end{pmatrix}, \\
	y(t) &= \begin{bmatrix} 1 & -1 \end{bmatrix}\begin{pmatrix} x_1(t) \\ x_2(t)\end{pmatrix}.
	\end{aligned}
\end{equation}
Choosing $F=[1 , -1 ]$ the Lipschitz condition~\eqref{Lipschitz} is satisfied as
\begin{equation*}
	\begin{aligned}
\norm{f_L(x)-f_L(\hat{x})} &= \| \sin(x_1-x_2)-\sin(\hat{x}_1-\hat{x}_2)\| \\
&\leq \norm{(x_1-x_2) - (\hat{x}_1-\hat{x}_2)} = \norm{\begin{bmatrix} 1& - 1 \end{bmatrix} \begin{pmatrix} x_1-\hat{x}_1 \\ x_2-\hat{x}_2 \end{pmatrix}} \\
	\end{aligned}
\end{equation*}
for all $x,\hat{x} \in \mathcal{X} = \R^2$.
The monotonicity condition~\eqref{Monotonicity} is satisfied with $\Theta = I_{q_M} = 1$, $\mu =2$ and $J=[0,1]$ since for all $x,z \in \hat{\mathcal{X}}=J\mathcal{X} = \R$ we have
\begin{equation*}
	\begin{aligned}
(z-x)\left( f_M(z)-f_M(x) \right) &= (z-x)\left(z+\exp(z)-x-\exp(x)\right) \\
 &= (z-x)^2 + \underset{\ge 0}{\underbrace{(z-x)\left(\exp(z)-\exp(x)\right)}} \ge \frac{\mu}{2} (z-x)^2.
	\end{aligned}
\end{equation*}
To satisfy the conditions of Theorem~\ref{Version-1} we choose $k=2$. A straightforward computation yields that conditions~\eqref{eq:Q-MI} and~\eqref{eq:EP-MI} are satisfied with the following matrices $\mathcal{P} \in ~\mathbb{R}^{(4+1)\times (2+2)}$, $\cK\in\R^{(2+2)\times (2+2)}$, $L_1 \in \mathbb{R}^{4 \times 2}$ and $L_2 \in \mathbb{R}^{1 \times 2}$ on the subspace $\mathcal{V}_{\left[[\mathcal{E},0],[\mathcal{A},\mathcal{B}]\right]}^*$ with $\delta =1$:
\begin{center}
\begin{tabular}{l l}
$\mathcal{P}=\frac{1}{10}\begin{bmatrix} 2& -2 & 0 & 0 \\ 0& 0 &0&0 \\0&0&0&0 \\ -2&3&0&0\\0&0&0&0 \end{bmatrix}$, &
$\mathcal{K}=\frac{1}{5}\begin{bmatrix}  *& *& 4& 10\\ *&*& -4& -10\\ *&*& 0& 0 \\ *& *& 0& 0\end{bmatrix}$,  \\
\newline \\
$\begin{bmatrix} L_1 \\ L_2 \end{bmatrix}$ = $\begin{bmatrix} 4 & 10 \\ 1 & 9 \\ 9 & 4 \\ 0 & 0 \\ 2 & 1 \end{bmatrix}$, &
$\mathcal{V}_{\left[[\mathcal{E},0],[\mathcal{A},\mathcal{B}]\right]}^* = \im\begin{bmatrix} 1 & 0 & 0 \\ 0&1 & 0 \\5 & -4 & 0 \\-11&9&0\\0&0&1\\-2&2&0 \end{bmatrix}$.
\end{tabular}
\end{center}
Then Theorem~\ref{Version-1} implies that a state estimator for~\eqref{Ex-Thm4_1_n+k<l+p} is given by
\begin{equation}
\label{Ex-Thm4_1_n+k<l+p-Estimator}
\begin{aligned}
	\ddt \begin{bmatrix} 1 & 0 \\ 1&1\\ 0&0 \\ 0&1 \end{bmatrix} \begin{pmatrix}
 z_1(t) \\ z_2(t) \end{pmatrix}
	&=\begin{bmatrix} 0& -3 \\-2&0\\\ 1 &-2 \\ 0&0 \end{bmatrix} \begin{pmatrix} z_1(t) \\ z_2(t) \end{pmatrix}
	+ \begin{bmatrix} 0&2\\1&-1\\0&1\\1&0\end{bmatrix}\begin{pmatrix}\sin(z_1(t) - z_2(t))
	\\  z_2(t) + \exp(z_2(t)) \end{pmatrix} +
	\begin{bmatrix} 4&10\\1&9\\9&4\\0&0\end{bmatrix} \begin{pmatrix}
	d_1(t) \\ d_2(t)
\end{pmatrix}	 \\
	0 &= \begin{bmatrix} 1 & -1 \end{bmatrix}\begin{pmatrix} z_1(t) \\ z_2(t)\end{pmatrix} - y(t) + \begin{bmatrix} 2&1\end{bmatrix} \begin{pmatrix}
	d_1(t) \\ d_2(t) \end{pmatrix}
	\end{aligned}
\end{equation}
Note, that $L_2$ is not invertible and thus the state estimator cannot be reformulated as a Luenberger type observer. Further, $n+k<l+p$ and therefore the pencil $s\cE-\cA$ is not square and hence in particular not regular; thus \eqref{Conditions-Version-2}~b) cannot be satisfied. In addition, for $F$ and $J$ in the present example, condition~\eqref{Conditions-Version-2}~e) does not hold (and is independent of~$k$), thus Theorem~\ref{Version-2} is not applicable here. A closer investigation reveals that for $k~=~l+p-n$ inequality~\eqref{eq:EP-MI} cannot be satisfied. We like to emphasize that $\cQ~<_{\mathcal{V}_{\left[[\mathcal{E},0],[\mathcal{A},\mathcal{B}]\right]}^*}~0$ but $\cQ~<~0$ does not hold on $\R^{n+k+q_L+q_M}~=~\R^6$.
\end{Ex}

The next example illustrates Theorem~\ref{Version-2}.

\begin{Ex}
We consider the DAE
\begin{equation}
\label{Ex-Thm4_3_fL_und_fM} %ex_Thm4_3_fL_und_fM_NEU.m
	\begin{aligned}
	\ddt \begin{bmatrix} 1 & -1\\ 0&0 \end{bmatrix} \begin{pmatrix}
 x_1(t) \\ x_2(t) \end{pmatrix}
	&=\begin{bmatrix}  -1&0 \\0 &1 \end{bmatrix} \begin{pmatrix} x_1(t) \\ x_2(t) \end{pmatrix}
	+ \begin{bmatrix} 2&-1\\-1&1\end{bmatrix}\begin{pmatrix}\sin\big(x_1(t)+x_2(t)\big) \\ x_1(t)+x_2(t)+\exp(x_1(t)+x_2(t))\end{pmatrix}, \\
	y(t) &= \begin{bmatrix}  1 & 1 \end{bmatrix}\begin{pmatrix} x_1(t) \\ x_2(t) \end{pmatrix}.
	\end{aligned}
\end{equation}
Similar to Example~\ref{Ex1} it can be shown that the monotonicity condition~\eqref{Monotonicity} is satisfied for $f_M(x) = x + \exp(x)$ with $J=[1,1]$, $\Theta = 1$ and $\mu=2$; the Lipschitz condition~\eqref{Lipschitz} is satisfied for $f_L(x_1,x_2) = \sin(x_1+x_2)$ with $F~=~[1,1]$.
%We know $\forall\, x,z \in \hat{\mathcal{X}}$
%\[
%\frac{\partial f_M}{\partial s} + \left( \frac{\partial f_M}{\partial s} \right)^\top \geq \mu I_{q_M} \Rightarrow (z-x)^\top\left( f_M(z,u,y)-f_M(x,u,y)\right) \geq \frac{\mu}{2} I_{q_M}.
%\]
%Here $\hat{\mathcal{X}} = J\mathcal{X} = [1,0,1]\mathcal{X}$,i.e., $s=x_1+x_3$. Hence the monotonicity condition~\eqref{Monotonicity} is satisfied with $\Theta = I_{q_M}$ and $\mu =2$:
%\[
%\frac{\partial f_M}{\partial s} + \left( \frac{\partial f_M}{\partial s} \right)^\top = 2\left( 1 + \exp(s) \right) \geq 2
%\]

Choosing $k=1$ a straightforward computation yields that conditions~\eqref{eq:Q-MI} and~\eqref{Conditions-Version-2}~a) are satisfied with $\delta=1.5$, the following matrices $\mathcal{P},\mathcal{K}~\in~\mathbb{R}^{(2+1)\times (2+1)}$, $L_1 \in \mathbb{R}^{2 \times 1}$ and $L_2 \in \mathbb{R}^{1 \times 1} = \R$ and subspaces $\mathcal{V}_{\left[[\mathcal{E},0],[\mathcal{A},\mathcal{B}]\right]}^*$,$\mathcal{V}_{\left[\mathcal{E},\mathcal{A}\right]}^*$ and $\mathcal{W}_{\left[\mathcal{E},\mathcal{A}\right]}^*$:
\begin{center}
\begin{tabular}{l l l}
$\mathcal{P}=\frac{1}{10}\begin{bmatrix}  1 & -1 & 0 \\  1 & 17 & 0 \\ 0 & 0 & 17 \end{bmatrix}$, &
$\mathcal{K}=\frac{1}{10}\begin{bmatrix}  *& *&  8\\ *& *& -134\\ *& *& 17 \end{bmatrix}$,  &
$\begin{bmatrix} L_1 \\ L_2 \end{bmatrix}$ = $\begin{bmatrix} 15 \\ -7 \\ 1 \end{bmatrix}$, \\
\newline \\
$\mathcal{V}_{\left[[\mathcal{E},0],[\mathcal{A},\mathcal{B}]\right]}^* = \im\begin{bmatrix} 1 & 0 & 0 \\0& 1 & 0 \\-1 & -1 & 0 \\0&0&1  \\ -7 & -8 & 1 \end{bmatrix}$, &
$\mathcal{V}_{\left[\mathcal{E},\mathcal{A}\right]}^* =
\im \begin{bmatrix} 8 \\ -7 \\ -1  \end{bmatrix}$,
&
$\mathcal{W}_{\left[\mathcal{E},\mathcal{A}\right]}^* =
\im \begin{bmatrix} 1 & 0 \\1&0\\0&1 \end{bmatrix}$.
\end{tabular}
\end{center}
Conditions~\eqref{Conditions-Version-2}~b)~--~e) are satisfied as follows:
\begin{enumerate}
\item[b)] $\det(s\cE-\cA)\neq 0$ and, using \cite[Prop.~2.2.9]{Berg14a}, the index of $s\cE-\cA$ is $\nu=k^*=1$, where $k^*$ is from Def.~\ref{Wong-sequences};
\item[c)] this holds since $G_L=[1/15,1/15]^\top$ and thus $\|F G_L\| <1$;
\item[d)] $J G_M$ is invertible since $G_M=-[1/15,1/15]^\top$ and  $\lambda_{max}(\Gamma)=-15<2=\mu$;
\item[e)] this condition is satisfied with e.g. $\alpha=1$ since $F=J$, and
\[
\frac{\alpha \|JG_L\|}{1- \|F G_L\|} \left(\sqrt{\frac{\max\{0,\lambda_{\rm max}(S)\}}{\mu - \lambda_{\rm max}(\Gamma)}} + \|(\Gamma-\mu I_{q_M})^{-1}(\tilde\Theta^\top - \mu I_{q_M})\|\right) = \frac{19}{221} < 1.
\]
\end{enumerate}
Then Theorem~\ref{Version-2} implies that a state estimator for system~\eqref{Ex-Thm4_3_fL_und_fM} is given by
\begin{equation} \label{Ex-Thm4_3_fL_und_fM-Estimator}
%\label{Ex-Thm4_3-estimator}
	\begin{aligned}
	\ddt \begin{bmatrix} 1 & -1\\ 0&0 \end{bmatrix} \begin{pmatrix}
 z_1(t) \\ z_2(t) \end{pmatrix}
	&=\begin{bmatrix}  -1&0 \\0 &1 \end{bmatrix} \begin{pmatrix} z_1(t) \\ z_2(t) \end{pmatrix}
	+ \begin{bmatrix} 2&-1\\-1&1\end{bmatrix}\begin{pmatrix}\sin\big(z_1(t)+z_2(t)\big) \\ z_1(t)+z_2(t)+\exp(z_1(t)+z_2(t))\end{pmatrix}\\
&\quad + \begin{bmatrix} 15 \\ -7 \end{bmatrix} d(t),	 \\
	0 &= \begin{bmatrix}  1 & 1 \end{bmatrix}\begin{pmatrix} z_1(t) \\ z_2(t) \end{pmatrix} - y(t) + d(t).
	\end{aligned}
\end{equation}
Straightforward calculations show that conditions \eqref{Conditions-Version-2}~a)~--~e) are satisfied, but condition~\eqref{eq:EP-MI} is violated; thus, Theorem~\ref{Version-1} is not applicable for $k=l+p-n=1$. The matrix~$L_2$ is invertible and hence the state estimator~\eqref{Ex-Thm4_3_fL_und_fM-Estimator} can be transformed as a standard Luenberger type observer. We  emphasize that $\mathcal{Q}<0$ does not hold on $\mathbb{R}^5$, i.e., the matrix inequality~\eqref{eq:Q-MI} on the subspace $\mathcal{V}_{\left[[\mathcal{E},0],[\mathcal{A},\mathcal{B}]\right]}^* \subseteq \mathbb{R}^5$ is a weaker condition.
\end{Ex}

The last example is an electric circuit where monotone nonlinearities occur, which is taken from~\cite{Riaz03}.

\begin{Ex}
Consider the electric circuit depicted in Fig.~\ref{Fig:RLC-circuit}, where a DC source with voltage~$\rho$ is connected in series to a linear resistor with resistance $R$, a linear inductor with inductance~$L$ and a nonlinear capacitor with the nonlinear characteristic
\begin{equation}\label{eq:cap-char}
    {q=g(v)=(v-v_0)^3-(v-v_0)+q_0},
\end{equation}
where $q$ is the electric charge and~$v$ is the voltage over the capacitor.

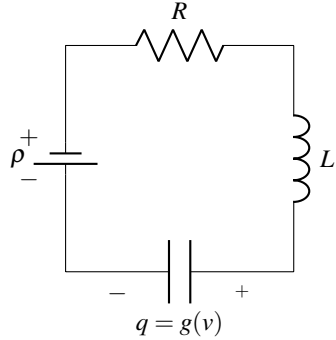
\begin{figure}[h!]
\begin{center}
\begin{circuitikz}[american, scale = 1.5]
  \draw (0,0)
  to[battery1, v^<=$\rho$] (0,2) % The voltage source
  to[R=$R$] (2,2) % The resistor
  to[L=$L$] (2,0) % The Inductor
  to[C, v^>=\text{$q=g(v)$}] (0,0); %The capacitor
\end{circuitikz}
\end{center}
\caption{Nonlinear RLC circuit} \label{Fig:RLC-circuit}
\end{figure}

Using the magnetic flux~$\phi$ in the inductor, the circuit admits the charge-flux description
\begin{equation}\label{eq:charge-flux}
	\begin{aligned}
\dot{q}(t) &= \frac{1}{L} \phi(t), \\
\dot{\phi}(t) &= -\frac{R}{L}\phi(t) - v(t) + \rho(t). \\
	\end{aligned}
\end{equation}
We scale the variables $q=$C~$\tilde{q}$, $\phi=$Vs~$\tilde{\phi}$, $v=$V~$\tilde{v}$ (where s, V and C denote the SI units for seconds, Volt and Coulomb, resp.) in order to make them dimensionless. For simulation purposes we set $\rho~=\rho_0~=~2$~V (i.e. $\rho$ trivially satisfies condition \eqref{Lipschitz}), $R=1\,\Omega$ and $L=0.5$~H, $\tilde{q}_0=\tilde{v}_0=1$. Then with $(x_1 , x_2 , x_3)^\top = \left(\tilde{q}-\tilde{q}_0,\tilde{\phi},\tilde{v}-\tilde{v}_0\right)^\top$ the circuit equations~\eqref{eq:cap-char},~\eqref{eq:charge-flux} can be written as the DAE
\begin{equation}
\label{RLC-system} %Example_RLC_RIAZA.m
	\begin{aligned}
	\ddt \begin{bmatrix} 1 & 0& 0\\0& 1&0 \\ 0&0&0 \end{bmatrix} \begin{pmatrix}
 x_1(t) \\ x_2(t) \\x_3(t) \end{pmatrix}
	&=\begin{bmatrix} 0& 2&0 \\0&-2 &-1 \\ -1 &0& -1 \end{bmatrix} \begin{pmatrix} x_1(t) \\ x_2(t) \\x_3(t) \end{pmatrix}
	+ \begin{bmatrix}0&0\\1&0\\0&1\end{bmatrix}\begin{pmatrix}  1 \\ x_3(t)^3\end{pmatrix} \\
	y(t) &= \begin{bmatrix}1&0&-1 \end{bmatrix}\begin{pmatrix} x_1(t) \\ x_2(t) \\x_3(t) \end{pmatrix},
	\end{aligned}
\end{equation}
where the output is taken as the difference $q(t)-v(t)$. Now, similar to the previous examples, a straightforward computation shows that Theorem~\ref{Version-2} is applicable and yields parameters for a state estimator for~\eqref{RLC-system}, which has the form
\begin{equation}
\label{RLC-estimator}
	\begin{aligned}
	\ddt \begin{bmatrix} 1 & 0& 0\\0& 1&0 \\ 0&0&0 \end{bmatrix} \begin{pmatrix}
 z_1(t) \\ z_2(t) \\z_3(t) \end{pmatrix}
	&=\begin{bmatrix} 0& 2&0 \\0&-2 &-1 \\ -1 &0& -1 \end{bmatrix} \begin{pmatrix} z_1(t) \\ z_2(t) \\z_3(t) \end{pmatrix}
	+ \begin{bmatrix}0&0\\1&0\\0&1\end{bmatrix}\begin{pmatrix}  1 \\ z_3^3(t)\end{pmatrix}
	+ \begin{bmatrix} -1 \\ 5 \\ 5	\end{bmatrix} d(t) \\
	0 &= \begin{bmatrix}1&0&-1 \end{bmatrix}\begin{pmatrix} z_1(t) \\ z_2(t) \\z_3(t) \end{pmatrix} - y(t) + 4 d(t).
	\end{aligned}
\end{equation}
Note that since $L_2=4$ is invertible, the given state estimator can be reformulated as an observer of Luenberger type with gain matrix $L=L_1L^{-1}_2$. As before we emphasize that $\mathcal{Q}<0$ is not satisfied on $\mathbb{R}^6$.\\
Note that this example also satisfies the assumptions of Theorem~\ref{Version-1} with $k=0$, i.e., the system copy itself serves as a state estimator (no innovation terms~$d$ are present).
\end{Ex}

%%%
\section{Comparison with the literature} \label{Comparison with other results}
%%%

We compare the results found in~\cite{Berg19,GuptToma18,LuHo06} to the results in the present paper. In~\cite[Thm.~2.1]{LuHo06} a way to construct an asymptotic observer of Luenberger type is presented. In the work~\cite{GuptToma18} a reduced-order observer design for non-square nonlinear DAEs is presented. An essential difference to Theorems~\ref{Version-1},~\ref{Version-2} and~\ref{Version-3} is the space on which the LMIs are considered. While in~\cite{GuptToma18,LuHo06} the LMI has to hold on~$\mathbb{R}^n$, the inequalities stated in the present paper as well as the inequalities stated in~\cite[Thm.~III.1]{Berg19} only have to be satisfied on a certain subspace where the solutions evolve in. While solving the LMIs stated in~\cite{GuptToma18,LuHo06} on the entire space~$\mathbb{R}^n$ is a much stronger condition, an advantage of this is that it can be solved numerically with little effort.

The second difference is that in~\cite{Berg19,LuHo06} the nonlinearity has to satisfy a Lipschitz condition of the form~\eqref{Lipschitz}, and the nonlinearity $f\in\cC^1(\mathbb{R}^r\to \mathbb{R}^r)$ in~\cite{GuptToma18} has to satisfy the generalized monotonicity condition $f'(s) + f'(s)^\top \geq \mu I_r$ for all $s\in\R^r$, which is less general than condition~\eqref{Monotonicity}, cf.~\cite{LiuDuan12}. In the present paper we allow the function $f = \binom{f_L}{f_M}$ to be a combination of a function~$f_L$ satisfying~\eqref{Lipschitz} and a function~$f_M$ satisfying~\eqref{Monotonicity}. Therefore the presented theorems cover a larger class of systems.

Furthermore,~\cite{LuHo06} consider square systems only, while in Theorems~\ref{Version-1},~\ref{Version-2} and~\ref{Version-3} we allow for any rectangular systems with $n\neq l$. Therefore, the observer design presented in the present paper is a considerable generalization of the work~\cite{LuHo06}.

Compared to~\cite[Thm. III.1]{Berg19}, we may observe that in this work the invertibility of a matrix consisting of system parameters and the gain matrices~$L_2$ and~$L_3$ is required. This condition as well as the rank condition is comparable to the regularity condition~\eqref{Conditions-Version-2} b) in the present paper. However, in the present paper we do not state explicit conditions on the gains, which are unknown beforehand and constructed out of the solution of~\eqref{eq:Q-MI}. Hence only the solution matrices~$\mathcal{P}$ and~$\mathcal{K}$ are required to meet certain conditions.

%%%
\section{Computational aspects} \label{Sec:CompAsp}
%%%

The sufficient conditions for the existence of a state estimator/asymptotic observer stated in Theorems~\ref{Version-1},~\ref{Version-2} and~\ref{Version-3} need to be satisfied at the same time, in each of them. Hence it might be difficult to develop a computational procedure for the construction of a state estimator based on these results, in particular since the subspaces $\mathcal{V}_{\left[[\mathcal{E},0][\mathcal{A},\mathcal{B}] \right]}^*$, $\mathcal{V}_{\left[\mathcal{E},\mathcal{A} \right]}^*$  and $\mathcal{W}_{\left[\mathcal{E},\mathcal{A} \right]}^*$ depend on the solutions~$\cP$ and~$\cK$ of~\eqref{eq:Q-MI}. The state estimators for the examples given in Section~\ref{examples} are constructed using ``trial and error'' rather than a systematic numerical procedure. The development of such a numerical method will be the topic of future research.

Nevertheless, the theorems are helpful tools in examining if an alleged observer candidate is a state estimator for a given system. To this end, we may set $\mathcal{K}H = \mathcal{P}^\top \hat{L}$ with given $\hat{L}$. Then~$\cA = \hat A + \hat L$ and the subspace to which~\eqref{eq:Q-MI} is restricted is independent of its solutions and hence~\eqref{eq:Q-MI} can be rewritten as a LMI on the space~$\R^{n^*}$, where $n^* = \dim \cV^*_{[[\cE,0],[\cA,\cB]]}$. This LMI can be solved numerically stable by standard MATLAB toolboxes like YALMIP~\cite{Lofb04} and PENLAB~\cite{FialKocv13}. For other algorithmic approaches see e.g.\ the tutorial paper~\cite{VanABraa00}.

%\printbibliography
\bibliographystyle{spmpsci}
%\bibliography{literature}
%\bibliography{MST}

\printindex

\end{document}